\documentclass[12 pt]{amsart}
\usepackage{amssymb,latexsym,amsmath,amscd,amsthm,graphicx, color}
\usepackage[all]{xy}
\usepackage{pgf,tikz}
\usepackage{mathrsfs}
\usepackage{cite}
\usetikzlibrary{arrows}
\definecolor{uuuuuu}{rgb}{0.26666666666666666,0.26666666666666666,0.26666666666666666}
\definecolor{xdxdff}{rgb}{0.49019607843137253,0.49019607843137253,1.}
\definecolor{ffqqqq}{rgb}{1.,0.,0.}

\definecolor{uuuuuu}{rgb}{0.26666666666666666,0.26666666666666666,0.26666666666666666}
\definecolor{qqwuqq}{rgb}{0.,0.39215686274509803,0.}
\definecolor{zzttqq}{rgb}{0.6,0.2,0.}
\definecolor{xdxdff}{rgb}{0.49019607843137253,0.49019607843137253,1.}
\definecolor{qqqqff}{rgb}{0.,0.,1.}
\definecolor{cqcqcq}{rgb}{0.7529411764705882,0.7529411764705882,0.7529411764705882}

\definecolor{uuuuuu}{rgb}{0.26666666666666666,0.26666666666666666,0.26666666666666666}
\definecolor{qqwuqq}{rgb}{0.,0.39215686274509803,0.}
\definecolor{zzttqq}{rgb}{0.6,0.2,0.}
\definecolor{xdxdff}{rgb}{0.49019607843137253,0.49019607843137253,1.}
\definecolor{qqqqff}{rgb}{0.,0.,1.}
\definecolor{cqcqcq}{rgb}{0.7529411764705882,0.7529411764705882,0.7529411764705882}

\theoremstyle{plain}

\newtheorem{theorem}[subsubsection]{Theorem}

\newtheorem{lemma}[subsubsection]{Lemma}

\newtheorem{prop1}[subsection]{Proposition}
\newtheorem{prop}[subsubsection]{Proposition}
\newtheorem{remark}[subsubsection]{Remark}
\newtheorem{defi}[subsubsection]{Definition}

\theoremstyle{definition}
\newtheorem{cor}[subsubsection]{Corollary}

\newtheorem{note}[subsubsection]{Note}
\newtheorem{example}[subsubsection]{Example}

\newcommand{\uu}{\cup}
\newcommand{\ii}{\cap}
\newcommand{\UU}{\bigcup}

\newcommand{\sci}{\subset}
\newcommand{\es}{\emptyset}
\newcommand{\set}[1]{\{#1\}}

\newcommand{\ga}{\alpha}
\newcommand{\gb}{\beta}

\newcommand{\gk}{\kappa}

\newcommand{\gn}{\nu}
\newcommand{\go}{\omega}

\newcommand{\gt}{\tau}

\newcommand{\tit}{\textit}

\newcommand{\C}[1]{\mathcal{#1}}
\newcommand{\D}[1]{\mathbb{#1}}

\newcommand{\te}{\text}

\begin{document}
To appear, Qualitative Theory of Dynamical Systems
\title{Canonical sequences of optimal quantization for condensation measures}

\author{Do\u gan \c C\"omez}
\address{Department of Mathematics \\
408E24 Minard Hall,
North Dakota State University
\\
Fargo, ND 58108-6050, USA.}
\email{dogan.comez@ndsu.edu}

\author{ Mrinal Kanti Roychowdhury}
\address{School of Mathematical and Statistical Sciences\\
University of Texas Rio Grande Valley\\
1201 West University Drive\\
Edinburg, TX 78539-2999, USA.}
\email{mrinal.roychowdhury@utrgv.edu}

\subjclass[2010]{28A80, 60Exx, 94A34.}
\keywords{Canonical sequence, optimal quantizers, quantization error, condensation measure, self-similar measure, quantization dimension, quantization coefficient}
\thanks{The research of the second author was supported by U.S. National Security Agency (NSA) Grant H98230-14-1-0320}

\date{}
\maketitle

\pagestyle{myheadings}\markboth{Do\u gan \c C\"omez and Mrinal Kanti Roychowdhury}{Canonical sequences of optimal quantization for condensation measures}

\begin{abstract}

We consider condensation measures of the form $P:=\frac 13  P\circ S_1^{-1}+ \frac 13  P\circ S_2^{-1}+ \frac 13  \nu $ associated with the system $(\mathcal{S}, (\frac 13, \frac 13, \frac 13), \gn) , $ where $\mathcal{S}=\{S_i\}_{i=1}^2 $ are contractions
and $ \nu$ is a Borel probability measure on $\mathbb R$ with compact support.
Let $D(\mu)$ denote the quantization dimension of a measure $\mu$ if it exists.  In this paper, we study self-similar measures $\gn$ satisfying $D(\gn )>\gk$, $D(\gn)<\gk$, and  $D(\gn)=\gk , $ respectively, where $\gk $ is the unique number satisfying
$[\frac13 (\frac{1}{5})^2]^{\frac{\gk}{2+\gk}}=\frac 12. $ For each case we construct two sequences $a(n)$ and $F(n)$, which are utilized in determining the optimal sets of $F(n)$-means and the $F(n)$th quantization errors for $P. $   We also show that for each measure $\gn$ the quantization dimension $D(P)$ of $P$ exists and satisfies $D(P)=\max\set{\gk, D(\gn)}. $  Moreover, we show that for $D(\gn)>\gk$, the $D(P)$-dimensional lower and upper quantization coefficients are finite, positive and unequal; and for $D(\gn)\leq \gk$, the $D(P)$-dimensional lower quantization coefficient is infinity.

\end{abstract}

\section{Introduction}

The problem of quantization of probability measures emanated from information theory and has been a subject of intense investigation in the past decades.
Rigorous mathematical treatment of basic quantization theory can be found in \cite{GL1}; for further results and applications one is referred to \cite{BW, G, GG, GLP, GN, KZ1, KZ2, P1, P2, Z1, Z2}. Let $P$ denote a Borel probability measure on $\D R^d$ with the Euclidean norm $\|\cdot\|, \ d\geq 1$. For a finite set $\ga \subset \D R^d, $ the value $V(P; \ga):=\int \min_{a \in \ga} \|x-a\|^2 dP(x) $ is often referred to as the \tit{cost} or \tit{distortion error} for $\ga$ with respect to $P$. Then, the $n$-th \textit{quantization
error} for $P$ is defined by
\[V_n:=V_n(P)=\inf\set{V(P; \ga) :\alpha \subset \mathbb R^d, \text{ card}(\alpha) \leq n}.\]
If $\int \| x\|^2 dP(x)<\infty , $ then there exists $\ga\sci \D R^d$ for which the infimum is achieved (see \cite{AW, GKL, GL, GL1}). A set $\ga$ for which the infimum is achieved and contains no more than $n$ points, i.e.,
$V_n=V(P; \ga)$, is called an \textit{optimal set of $n$-means} or an \textit{optimal set of $n$-quantizers}.  Indeed, if the support of $P $ is an infinite set, then an optimal set of $n$-means always contains exactly $n$ elements (see Theorem 4.12 \cite{GL1} and Theorem 2.4 \cite{GL3}).
The limit
$D(P):=\lim_{n\to \infty}  \frac{2\log n}{-\log V_n(P)}, $
if exists, is called the \tit{quantization dimension} of $P$, which measures the speed at which the specified measure of the error tends to zero as $n \to \infty . $ It turns out that determining the optimal sets of $n$-means is much more difficult than calculating the quantization dimension of a measure \cite{DR, GL2, R1, R2, RR}.  If $D(P):=s$ exists, we are further concerned with the $s$-dimensional lower and upper quantization coefficients defined by $\liminf_{n\to \infty} n^{\frac 2 s} V_n(P)$ and $\limsup_{n\to \infty} n^{\frac 2 s} V_n(P)$, respectively. These two quantities provide more accurate information for the asymptotics of the quantization error than the quantization dimension. Given a finite subset $\ga\sci \D R^d$, the \tit{Voronoi region} $M(a|\ga)$ generated by $a\in \ga$ is defined as the set of points $x \in \D R^d$ such that $a$ is the nearest point to $x$ than to all other elements in $\ga$.
If $\ga$ is an optimal set and $a\in\ga$, then $a$ is the \tit{conditional expectation} of the random variable $X$ given that $X$ takes values in the Voronoi region of $a$ \cite{GG, GL1}. Such a point $a$ is also refereed to as the \tit{centroid} of the Voronoi region $M(a|\ga)$ with respect to $P$.

Let $S_i : \D R^d \to \D R^d$ for $i=1, 2, \cdots, N$ be contracting similarities, $\bold{p}=(p_0, p_1, \cdots, p_N)$ be a probability vector, and $\gn$ be a Borel probability measure on $\D R^d$ with compact support. A probability measure $P$ on $\D R^d$ such that $P=\sum_{j=1}^N p_j P\circ S_j^{-1} +p_0\gn $
is called an \tit{inhomogeneous self-similar measure} associated with the system $(\mathcal{S}, \bold{p} , \gn), $ where $\mathcal{S}=\{S_i\}_{i=1}^N . $  For such an inhomogeneous self-similar measure $P$, if $C$ is the support of $\gn,$ then the support of $P$ is equal to the unique nonempty compact set $K:=K_C \subset \D R^d$ satisfying $K=\mathop{\uu}\limits_{j=1}^N S_j(K)\uu C.$
For details about inhomogeneous self-similar sets and measures one can see \cite{OS}. Following \cite{B, L}, we call $(\mathcal{S},\bold{p} , \gn)$ a \tit{condensation system}. The measure $P$ is also called the \tit{condensation measure} or the \tit{attracting measure} for $(\mathcal{S}, \bold{p}, \gn)$, and the set $K$, which is the support of the measure $P$, is called the \tit{attractor} of the system.

Consider a condensation measure $P$ such that $P=\frac 1 3 P\circ S_1^{-1}+\frac 13 P\circ S_2^{-1}+\frac 13 \gn$, where $\gn$ is a self-similar measure on $\D R$ and $S_1, S_2$ are similarity mappings on $\D R$ given by $ S_1(x)= c x, \ S_2(x)= c x+ r, $ with $0<c \leq \frac{1}{3} $ and $c+r=1 . $  Let $ \gk $ be the number satisfying $(\frac{c^2}{3})^{\frac{\gk}{2+\gk}}=\frac{1}{2}, $ which we will call as the {\it critical value} of the condensation system $(\mathcal{S}, \bold{p}, \gn). $
In this article our aim is two-fold.  The main goal is to study the quantization for condensation measures associated to the system
$(\mathcal{S}, \bold{p}, \gn); $ in particular, we show that there exist two sequences $\set{a(n)}_{n\geq 1}$ and $\set{F(n)}_{n\geq 1}$, which we call as \tit{canonical sequences}, that are instrumental in this study.  With the help of canonical sequences, for a variety of self-similar measures $\gn , $ we obtain closed formulas for the optimal sets of $F(n)$-means and the $F(n)$th quantization errors for the condensation measures $P$ for all $n\geq 1$.  Once the optimal sets of $F(n)$-means are known, we develop a simple method to determine the optimal sets of $n$-means for all $n\in \D N ,$ and calculate quantization dimension $D(P)$ and the $D(P)$-dimensional quantization coefficients for $P.$
Hence, we furnish the {\bf complete quantization program} for the condensation systems under consideration, which was not done before in the literature.
The second aim is to investigate the relationship between the quantization dimension $D(P)$ and the $D(P)$-dimensional quantization coefficients as the measure $\nu $ changes.
It turns out that $D(P)$ satisfies the relation $D(P)=\max\set{\gk, D(\gn)}. $   Furthermore, we determine that for $D(\gn)>\gk$, the $D(P)$-dimensional lower and upper quantization coefficients are finite, positive and unequal.  On the other hand, for $D(\gn)< \gk$, and  $D(\gn)=\gk$, the $D(P)$-dimensional lower quantization coefficients are infinity.

When $p_0=0$, the inhomogeneous self-similar measure on $\mathbb{R}^d $ defined above by $P=\sum_{j=1}^N p_j P\circ S_j^{-1} +p_0\nu $ reduces to the self-similar measure $P=\sum_{j=1}^N p_j P\circ S_j^{-1}$.  Quantization dimension for such self-similar measures and self-conformal measures were determined by Graf-Luschgy \cite{GL4} and Lindsay-Mauldin \cite{LM}, respectively.  Following these, the quantization dimensions were determined for many fractal probability measures \cite{R4, R5, R6, R7}.  The optimal sets of $n$-means and the $n$th quantization errors for the (standard) Cantor self-similar measure were determined by Graf-Luschgy \cite{GL2}.  Due to their intricate structures, condensation measures which are more general than self-similar measures, were not studied widely in the literature; in particular, this is the case for the features investigated in this article.
To the best of the authors' knowledge, the current article is the first study of the optimal sets of $n$-means and the $n$th quantization errors for condensation measures, which include all (Cantor) self-similar measures as a special case.  Hence, our results are significant generalization of those in \cite{GL2}. The novelty in obtaining the quantization for condensation measures is the introduction of canonical sequences and the order $\succ $ for the associated optimal sets. As a by-product  we also derive closed formulas for the quantization errors involved at each step, which lead to direct calculation of the quantization dimensions and study the quantization coefficients for the probability distributions involved.

Techniques utilized in the article can be applied or can further be improved to investigate the optimal sets of $n$-means and the $n$th quantization errors for many other (fractal-based) singular probability measures. All these results are new and, while providing some insight into the behavior of such systems, they also bring up several questions for further inquiry.  We outline some of these at the end of the paper.
The points in optimal sets, being the centroids of their Voronoi regions, are an evenly spaced distribution of sites
in the domain with minimum distortion error with respect to a given probability measure. Such settings are frequently surface in many fields, such as numerical probability \cite{BC, PPP}, clustering, data compression, optimal mesh generation, signal processing, cellular biology, optimal quadrature, and geographical
optimization \cite{DFG, OBSC, HCHSVH, KKR, S, ABDHW, LZSC}. Therefore, the result of the paper have potential for being very useful in addressing problems in these fields.

The arrangement of the paper is as follows:  In Section~\ref{first}  basic definitions, lemmas and propositions are developed. In order to bring some transparency to the arguments and reveal the connections between the quantization dimension $D(\nu)$ and the critical value $\kappa $, in Subsection~\ref{subfirst2} four important cases that will be focus of the article are outlined. Section~\ref{second} provides a thorough investigation of the quantization for the condensation measure $P$ with self-similar measure $\nu $ satisfying $D(\nu)>\kappa $. Also, utilizing canonical sequences, we have determined all the optimal sets of $n$-means and the associated $n$th quantization errors, and calculated the quantization dimension of $P$ (Theorem 3.4.1). Furthermore, it is also shown that the quantization coefficient does not exist, and the lower and the upper quantization coefficients are finite and positive (Theorem 3.4.2). Section~\ref{third} is devoted to the investigation of quantization for the condensation measure $P$ with self-similar measure $\nu$ satisfying $D(\nu)<\kappa . $  Using closed formulas obtained for quantizations errors, we calculate the quantization dimension of $P$ (Theorem 4.4.1), and show that the quantization coefficient is infinity (Theorem 4.4.2). In Section~\ref{fourth}, we have considered two condensation measures: one with $D(\nu)>\kappa$ and one with $D(\nu)<\kappa$; and show that the quantization coefficients in both the cases are infinity. The results in all the above sections lead us to some observations and remarks outlined in Concluding Remarks ~\ref{conr} which also contains some open problems to be investigated.

In the sequel, all the arguments will be given for $\bold{p}=(\frac 13, \frac 13, \frac 13) $ for simplicity.

\section{Preliminaries}\label{first}

\subsection{Basic definitions, lemmas and propositions} \label{subfirst1}

Let $P$ be the condensation measure associated with the condensation system $(\mathcal{S}, \bold{p}, \gn)$,
where $\bold{p}=(\frac 13, \frac 13, \frac 13) $ and $\mathcal{S}$ is as defined above.  Consider the self-similar measure $\gn$ given by $\gn=\frac 12 \gn \circ T_1^{-1}+\frac 12\gn\circ T_2^{-1}$, where $T_1(x)=s x+ a(1-s)$ and $T_2(x)=sx+ b(1-s)$ for all $x\in \D R ,\ 0<s\leq \frac 13,$ and $a=\frac{1+c}{3}, \ b=\frac{2-c}{3}. $  These values are to ensure that the associated Cantor sets are disjoint.

Let $I=\set{1, 2}$. By a word $\go$ of length $k$ over the alphabet $I$, we mean $\go:=\go_1\go_2\cdots \go_k \in I^k ; $ a word of length zero is called the empty word and is denoted by $\es$.  $I^\ast$ will denote the set of all words over the alphabet $I$ including the empty word $\es$.
For $\go , \gt \in I^\ast ,$ their concatenation is denoted by  $ \go\gt ;$ i.e., if $\go:=\go_1\go_2\cdots \go_k$ and $\gt:=\gt_1\gt_2\cdots \gt_\ell,$ then $\go\gt:=\go_1\cdots \go_k\gt_1\cdots \gt_\ell$.
Set $J:=[0, 1]$ and $L:=[a, b]$.  For $\go=\go_1\go_2 \cdots\go_k \in I^k$, set $S_\go:=S_{\go_1} \circ \cdots \circ S_{\go_k}$, $T_\go:=T_{\go_1} \circ \cdots \circ T_{\go_k}$, $J_\go:=S_\go(J)$, and $L_\go:=S_\go(L)$. For $\go =\es, \  S_\es$ is the identity map on $\D R$, $J_\es=J$ and $L_\es=L$. If $C$ is the support of $\gn$, then \[C:=\mathop{\ii}\limits_{k\geq 0} \mathop{\uu}\limits_{\go\in I^k} T_\go([a,b]).\]
Iterating $P=\frac 13 \mathop{\sum}\limits_{j=1}^2 P\circ S_{j}^{-1} + \frac 13\gn$ and $\gn=\frac 12\gn\circ T_1^{-1}+\frac 12\gn\circ T_2^{-1}$, we have
\begin{equation} \label{eq1} P=\frac 1{3^n} \mathop{\sum}\limits_{|\go|=n} P\circ S_{\go}^{-1} +\mathop{\sum}\limits_{k=0}^{n-1}\frac 1{3^{k+1}} \mathop{\sum}\limits_{|\go|=k} \gn \circ S_{\go}^{-1}, \te{ and } \gn=\frac 1{2^k} \mathop{\sum}\limits_{\go \in I^k} \gn\circ T_\go^{-1}, \ \text{for all}\ k\geq 1 .
\end{equation}
The measure $P$ is `symmetric' about the point $\frac 12$, i.e., if two intervals of equal lengths are equidistant from the point $\frac 12$, then they have the same $P$-measure. For $n\geq 1$, $\ga_n:=\ga_n(P)$ will denote an optimal set of $n$-means with respect to $P$, and $\ga_n(\gn)$ will represent an optimal set of $n$-means with respect to the self-similar measure $\gn$. Similarly, $V_n:=V_n(P)$ and $V_n(\gn)$ represent the $n$th quantization error with respect to $P$ and $\gn , $ respectively. By $P|_L$, we denote the conditional probability measure on $L$, i.e., for any Borel $B\sci \D R$,
\begin{equation} \label{eq345}
P|_L(B)=\frac{P(B\ii L)}{P(L)}.
\end{equation}
Notice that $P|_L=\gn$.  Using equation~\eqref{eq1}, we deduce the following lemma.

\begin{lemma} \label{lemma1} Let $g: \D R \to \D R^+$ be Borel measurable and $n\in \D N$. Then,
\[\int g \,dP=\frac 1 {3^n} \sum_{|\go|=n}\int (g\circ S_\go) \,dP+\sum_{k=0}^{n-1} \frac 1 {3^{k+1}}\sum_{|\go|=k} \int (g\circ S_\go) \,d\gn.\]
\end{lemma}

\begin{lemma} Let $K$ be the support of the condensation measure. Then, for any $n\geq 1$,
\[K\sci (\mathop{\uu}\limits_{\go\in I^n} J_{\go})\UU  \Big (\mathop{\uu}\limits_{k=0}^{n-1}(\mathop{\uu}\limits_{\go\in I^k} L_\go)\Big )\sci J.\]
\end{lemma}

\begin{proof}  Notice that $J_1\uu L\uu J_2\sci J$, $J_{11}\uu L_1\uu J_{12}\sci J_1$, and $J_{21}\uu L_2\uu J_{22}\sci J_2$. In fact, for any $k\geq 1$, if $\go \in I^k$, then $J_{\go1} \uu L_{\go}\uu J_{\go2}\sci J_\go$. Again, notice that for any $\go\in I^\ast$, $J_{\go1}\uu J_{\go2}\sci J_{\go}$, and the intervals $L_{\go1}, L_{\go2}, L_\go$ are disjoint. Thus, it follows that
\[(\mathop{\uu}\limits_{\go\in I^n} J_{\go})\UU  \Big (\mathop{\uu}\limits_{k=0}^{n-1}(\mathop{\uu}\limits_{\go\in I^k} L_\go)\Big )\sci J.\]
The sets being disjoint, we have
\begin{align*}
&P\Big((\mathop{\uu}\limits_{\go\in I^n} J_{\go})\UU  \Big (\mathop{\uu}\limits_{k=0}^{n-1}(\mathop{\uu}\limits_{\go\in I^k} L_\go)\Big )\Big)=P(\mathop{\uu}\limits_{\go\in I^n} J_{\go})+P\Big (\mathop{\uu}\limits_{k=0}^{n-1}(\mathop{\uu}\limits_{\go\in I^k} L_\go)\Big )\\
&=\sum_{\go\in I^n} P(J_\go)+\sum_{k=0}^{n-1} \sum_{\go \in I^k} P(L_\go)=\sum_{\go\in I^n} \frac 1{3^n} +\sum_{k=0}^{n-1} 2^k\cdot \frac 1{3^{k+1}}=1.
\end{align*}
Again, $P(K)=1$ and $K$ is the support of $P$. Hence, $K\sci (\mathop{\uu}\limits_{\go\in I^n} J_{\go})\UU  \Big (\mathop{\uu}\limits_{k=0}^{n-1}(\mathop{\uu}\limits_{\go\in I^k} L_\go)\Big )$.
\end{proof}

Let $E(\gn)$ and $W:=V(\gn)$ represent the expected value and the variance of $\gn , $ respectively.

\begin{lemma} \label{lemma2} For the self-similar measure $\gn$ we have
\begin{itemize}
\item[(i)] $E(\gn)=\frac 12$ and $W= \frac{(1-s)(1-2c)^2}{36(1+s)} $,
\item[(ii)] for any $x_0\in \D R$, $\int(x-x_0)^2 d\gn=(x_0-\frac 12)^2+V(\gn)$.
\end{itemize}
\end{lemma}

\begin{proof} Since
$ \int x d\gn=\frac 12 \int [ s x+a(1-s)] d\gn+\frac 12\int [sx+b(1-s)] d\gn $ and $a+b=1, $ we have $E(\gn)=\int x d\gn=\frac 12 .$  Moreover,
$$ \aligned \int x^2 d\gn & =\frac 12 \int [sx+ a(1-s)]^2 d\gn+\frac 12\int [sx+b(1-s)]^2 d\gn \\
& =s^2 \int x^2 d\gn + s(1-s)\int x d\gn + \frac{1}{2} (a^2+b^2)(1-s)^2 \int d\gn, \endaligned $$
which implies that $\int x^2 d\gn =\frac{s+(a^2+b^2)(1-s)}{2(1+s) }. $
Hence,
$$ W:=V(\gn) =\int x^2 d\gn-(\int x d\gn)^2= \frac{s+(a^2+b^2)(1-s)}{2(1+s) }-\frac 14
= \frac{(1-s)(1-2c)^2}{36(1+s)}.  $$
For any $x_0\in \D R$, $\int(x-x_0)^2 d\gn=(x_0-\frac 12)^2+V(\gn)$ follows from the standard probability arguments.
\end{proof}

Let $E(P)$ and $V(P)$ represent the expected value and the variance of $P ,$ respectively.
\begin{lemma} \label{lemma3} For the condensation measure $P$, we have
$$ E(P)=\frac 12, \  V(P)= \frac{(1-c)^2}{2 (3-2c^2)} + V(\nu),  $$
and, for any $x_0\in \D R$, $\int(x-x_0)^2 dP=(x_0-\frac 12)^2+V(P)$.
\end{lemma}
\begin{proof} It is straightforward to see that $E(P)=\frac 1 2$.   Now, using \eqref{eq1}, we have
\begin{align*}
\int x^2 dP & =\frac 1 3 \int  (S_1(x))^2 dP+\frac 13 \int (S_2(x))^2 dP +\frac 13 \int x^2 d\gn\\
&=\frac 1 3  \int (c x)^2 dP +\frac 13  \int (cx+r)^2 dP+\frac 13 \int x^2 d\gn;  \end{align*}
hence,
$ (1-\frac{2c^2}{3} )\ \int x^2 dP =\frac{1}{3} \Big(cr+r^2 +\int x^2 d\gn \Big). $  Since $c+r=1, $ this implies that
$$\int x^2 dP=\frac{1}{3-2c^2} (r+\int x^2 d\gn ). $$
Thus, it follows that
$$ \aligned V(P) & = \int x^2 dP- \frac{1}{4}= \frac{1}{3-2c^2} \Big( r+ \int x^2 d\gn \Big) - \frac{1}{4} \\
&=\frac{1}{3-2c^2} \Big( r+V(\nu)+ \frac{c^2-1}{2} \Big)
=  \frac{(1-c)^2}{2 (3-2c^2)} + V(\nu) . \endaligned $$

For any $x_0 \in \D R$, $\int(x-x_0)^2 dP=(x_0-\frac 12)^2+V(P)$ follows from the standard probability arguments.
\end{proof}

\medskip

\begin{note}
Let $\go\in I^k$, $k\geq 0$, and let $X$ be the random variable with probability distribution $P$. Then, by equation~\eqref{eq1}, it follows that $P(J_\go)=\frac 1{3^k}$, $P(L_\go)=\frac 1{3^{k+1}}$,
\begin{align*} & E(X : X \in J_\go)=\frac 1{P(J_\go)} \int_{J_\go} x dP=\int x d(P\circ S_\go^{-1})=\int S_\go(x) dP=S_\go(\frac 12), \te{ and }\\
&E(X : X \in L_\go)=\frac 1{P(L_\go)} \int_{L_\go} x dP=\int x d(\gn \circ S_\go^{-1})=\int S_\go(x) d\gn=S_\go(\frac 12).
\end{align*}
For any $x_0 \in \mathbb R$,
\begin{align} \label{eq4} \left\{\begin{array}{cc}  \int_{J_\go} (x-x_0)^2 dP(x) =\frac 1{3^k} \Big(c^{2k} V  +(S_\go(\frac 12)-x_0)^2\Big), \\
\int_{L_\go} (x-x_0)^2 dP(x) =\frac 1{3^{k+1}} \Big(c^{2k} W  +(S_\go(\frac 12)-x_0)^2\Big).&\end{array}\right.
\end{align}
On the other hand, for any $x_0\in \D R$, any $\go\in I^k$, $k\geq 0$,
\begin{align} \label{eq40} \int_{T_\go(L)} (x-x_0)^2 d\gn(x) =\frac 1{2^k} \Big(s^{2k} W  +(T_\go(\frac 12)-x_0)^2\Big).\end{align}
\end{note}

\begin{remark}
By Lemma~\ref{lemma3}, it follows that the optimal set of one-mean for the condensation measure $P$ consists of the expected value $\frac 12$ and the corresponding quantization error is the variance $V(P)$ of $P$, i.e., $V(P)=V_1(P)$. Notice that by `the variance of $P$' it is meant the variance of the random variable $X$ with distribution $P$.
\end{remark}


\begin{prop} (see \cite{GL2}) \label{prop0}  For $n\in \D N$ with $n\geq 2$ let $\ell(n)$ be the unique natural number with $2^{\ell(n)} \leq n<2^{\ell(n)+1}$. Let $\ga_n(\gn)$ be an optimal set of $n$-means for $\gn$, i.e., $\ga_n(\gn)\in \C C_n(\gn)$. Then,
\[\ga_n(\gn)=\set{T_\go(\frac 12) : \go \in I^{\ell(n)} \setminus \tilde I} \uu \set{T_{\go 1}(\frac 12) : \go \in \tilde I} \uu \set {T_{\go 2}(\frac 12) : \go \in \tilde I}\] for some $\tilde I\sci I^{\ell(n)}$ with card$(\tilde I)=n-2^{\ell(n)}$.
Moreover,
\begin{equation*}\label{eq11}
V_n(\gn)=\int\mathop{\min}\limits_{a\in\ga_n(\gn)} (x-a)^2 d\gn = (\frac{s^2}{2})^{\ell(n)} W\Big(2^{\ell(n)+1}-n+s^2(n-2^{\ell(n)})\Big).
\end{equation*}
\end{prop}

The following lemma is straightforward; hence, we will state it without proof.

\begin{lemma}  \label{lemma4} Let $\ga$ be an optimal set of $n$-means for the condensation measure $P$. Then, for any $\go \in I^\ast$, the set $S_\go(\ga):=\set{S_\go(a) : a \in \ga}$ is an optimal set of $n$-means for the image measure $P\circ S_\go^{-1}$. Conversely, if $\gb$ is an optimal set of $n$-means for the image measure $P\circ S_\go^{-1}$, then $S_\go^{-1}(\gb)$ is an optimal set of $n$-means for $P$.
\end{lemma}

\begin{lemma}
If $\ga_n(\gn)$ is an optimal set of $n$-means for $\gn , $ then, for any $\go \in I^k$, $k\geq 0$, $S_\go(\ga_n(\gn))$ is an optimal set of $n$-means for the measure $\gn\circ S_\go^{-1}$. Moreover,
\[\int_{L_\go}\min_{a\in S_\go(\ga_n(\gn))}(x-a)^2 dP=\frac{c^{2k}}{3^{k+1}}  V_n(\gn).\]
\end{lemma}
\begin{proof}Let $\ga_n(\gn)$ be an optimal set of $n$-means for $\gn$. Then, $S_\go(\ga_n(\gn))$ is an optimal set of $n$-means for the image measure $\gn\circ S_\go^{-1}$ follows from Lemma~\ref{lemma4}. Now, using \eqref{eq1} and Proposition~\ref{prop0},
\begin{align*}
&\int_{L_\go}\min_{a\in S_\go(\ga_n(\gn))}(x-a)^2 dP=\frac {1}{3^{k+1}} \int_{L_\go}\min_{a\in S_\go(\ga_n(\gn))}(x-a)^2 d(\gn\circ S_\go^{-1})\\
&=\frac {1}{3^{k+1}} \int_{L}\min_{a\in S_\go(\ga_n(\gn))}(S_\go(x)-a)^2 d\gn=\frac {1}{3^{k+1}}\ c^{2k} \int_{L}\min_{a\in \ga_n(\gn)}(x-a)^2 d\gn=
\frac{c^{2k}}{3^{k+1}} V_n(\gn),
\end{align*}
which completes the proof of the lemma.
\end{proof}

Next, we will determine the optimal sets of 2 and 3 means which will provide the base needed to determine the optimal sets of $F(n)$-means and the $F(n)$th quantization errors for a canonical sequence $\set{F(n)}_{n\geq 1}$.

\begin{prop} \label{prop1}
Let $\ga:=\set{a_1, a_2}$ be an optimal set of two-means with $a_1<a_2$. Then, $a_1=\frac{2}{3} \Big(S_1 (\frac{1}{2})+\frac{1}{2} T_1 (\frac{1}{2}) \Big), $ $a_2=\frac{2}{3} \Big(S_2 (\frac{1}{2})+\frac{1}{2} T_2 (\frac{1}{2})\Big), $ and the corresponding quantization error is
$$V_2=2 \Big( \int_{J_1} (x-a_1)^2 dP + \int_{L_1} (x-a_1)^2 dP  \Big). $$
\end{prop}
\begin{proof}
Due to symmetry of the condensation measure $P$ with respect to the midpoint $\frac 12$, we can assume that if $\set{a_1, a_2}$ is an optimal set of two-means with $a_1<a_2$, then $a_1=E(X : X \in [0, \frac 12])$ and $a_2=E(X : X \in [\frac 12, 1])$. Since $P(J_1\uu T_1(L))=P(J_1)+P(T_1(L))=\frac 13+\frac 13\gn(T_1(L))=\frac 13+\frac 16=\frac 12$, we have
\begin{align*}
a_1&=E(X : X \in J_1\uu T_1(L))=\frac{1}{P(J_1\uu T_1(L))}\Big(\int_{J_1} x dP+\int_{T_1(L)}x dP\Big)\\
&=2\Big(\frac 13 S_1(\frac 12) +\frac 13\int_{T_1(L)} x d\gn\Big)=\frac{2}{3} \Big( S_1(\frac 12) +\frac 12 T_1(\frac 12)\Big).
\end{align*}
Similarly, $a_2=\frac{2}{3} \Big( S_2(\frac 12) +\frac 12 T_2(\frac 12)\Big)=1-a_1.$
The  corresponding quantization error is given by
$$ V_2 =\int\min_{a\in \ga} (x-a)^2 dP=2\Big(\int_{J_1}(x-a_1)^2 dP+\int_{T_1(L)} (x-a_1)^2 dP\Big)  ,  $$
which completes the proof.
\end{proof}

It should be observed that, by symmetry, in the proposition above we also have
$$V_2 =\int\min_{a\in \ga} (x-a)^2 dP=2\Big(\int_{J_2}(x-a_2)^2 dP+\int_{T_2(L)} (x-a_2)^2 dP\Big) . $$

The proof of the following lemma is straightforward.

\begin{lemma} \label{lemma0} Let $\go \in I^k$ for $k\geq 0$. Then,
\[\int_{J_{\go 1}\uu S_{\go}[a, \frac 12]}\Big(x-S_{\go}(a_1) \Big)^2 dP=\frac {c^{2k}}{3^{k}}  \frac 12 V_2=\int_{S_{\go}[\frac 12, b]\uu J_{\go 2}}\Big(x-S_{\go}(a_2)\Big)^2 dP.\]
\end{lemma}

From the above lemma we deduce the following corollary.
\begin{cor} \label{cor1}
Let $\go \in I^k$ for  $k\geq 0$. Then, for any $a\in \D R$,
\begin{align*}
&\int_{J_{\go 1}\uu S_{\go}(T_1(L))}(x-a)^2 dP=\frac 1{3^k}\Big(\frac{c^{2k}}{2} V_2+\frac 12(S_\go(a_1)-a)^2\Big), \te{ and } \\
&\int_{S_{\go}(T_2(L))\uu J_{\go 2}}(x-a)^2 dP=\frac 1{3^k}\Big(\frac{c^{2k}}{2} V_2+\frac 12(S_\go(a_2)-a)^2\Big).
\end{align*}
\end{cor}

\medskip

\begin{prop}\label{prop2}
Let $\ga:=\set{a_1, a_2, a_3}$ be an optimal set of three-means with $a_1<a_2<a_3$. Then, $a_1=S_1(\frac 12)=\frac {c}{2}$, $a_2=\frac 12$, and $a_3=S_2(\frac 12)=1- \frac {c}{2}$. The corresponding quantization error is $V_3=\frac{1}{3} (2c^2 V +W ) . $
\end{prop}

\begin{proof}  Since both $\nu $ and $P$ are symmetric uniform measures with support $K, $ we can assume that $a_1 \in J_1, \ a_2 \in L $ and $a_3=1-a_1. $  Hence, by (3), the quantization error due to this set of three points $\gb:=\set{a_1, a_2, a_3}$
is
$$ \aligned & \int\min_{a\in \gb} (x-a)^2 dP=\int_{J_1}(x-a_1)^2 dP+\int_{L}(x-a_2)^2 dP+\int_{J_2}(x-1+a_1)^2 dP\\
&=\frac 13 \Big[ c^2 V +(S_1(\frac 12)-a_1)^2 \Big]
+\frac 13 \Big[W + (\frac 12 -a_2)^2 \Big]
+\frac 13 \Big[c^2 V +(S_2(\frac 12)-1+a_1)^2 \Big]\\
& = \frac 13 \Big[2 c^2 V +W + 2 (\frac{c}{2} -a_1)^2 +(\frac 12 -a_2)^2 \Big]\\
& = \frac 13 \Big(2 c^2 V +W+ \frac{c^2}{2}+\frac{1}{4} \Big)
+\frac 13 \Big(a_1^2 +a_2^2 -c a_1 -a_2 \Big). \endaligned
$$
The function $f(a_1, a_2)=a_1^2 +a_2^2 -c a_1 -a_2 $ attains its minimum at $a_1=\frac{c}{2} $ and $a_2=\frac 12 $ with minimum value $-\frac{1}{3}(\frac{c^2}{2}+\frac{1}{4}). $  Therefore, since $V_3$ is the quantization error for three-means, we have $V_3 =  \frac 13 (2 c^2 V +W) . $
\end{proof}


\begin{prop} \label{prop00021}
Let $\ga$ be an optimal set of $n$-means for $n\geq 3$ such that $\ga\ii J_1\neq \es$, $\ga\ii J_2\neq \es$, and $\ga\ii L\neq \es$. Then, for $i=1, 2$,  the Voronoi region of any point in $\ga\ii J_i$ does not contain any point from $L$, and the Voronoi region of any point in $\ga\ii L$  does not contain any point from $\ga\ii J_i$.
\end{prop}

\begin{proof}
Let $\ga:=\set{a_1, a_2, \cdots, a_n}$ be an optimal set of $n$-means for $n\geq 3$ such that $0<a_1<a_2<\cdots<a_n<1$. Let $j=\max\set {1\leq i\leq n : a_i\in \ga\ii J_1}$. Then, by the hypothesis, $a_j\leq  c$. Suppose that the Voronoi region of $a_j$ contains points from $L$. Then, $\frac 12 (a_j+a_{j+1})> a$ implying that
$$a_{j+1}> 2a-a_j\geq  2a-c=b . $$
This leads to a contradiction because $\ga\ii L\neq\es$. Hence, the Voronoi region of any point in $\ga\ii J_1$ does not contain any point from $L$.  The rest of the statements are proved similarly.
\end{proof}

\bigskip

\subsection{Set-up for optimal sets of $n$-means for $n\geq 4$} \label{subfirst2}

As observed in the previous subsection, the interaction of the measures  $\nu $ and $P $ leads to rather intricate arguments.  In order to bring some transparency to the arguments and reveal the connections between the quantization dimension $D(\nu)$ and the critical value $\kappa $ we will proceed by considering some special cases of the measures $P $ and $\nu. $

Throughout the rest of the article we will assume that $c=\frac{1}{5}. $  Consequently, $a=\frac{2}{5}, \ b=\frac{3}{5} $ and
the critical value will be $\kappa =\frac{2 \log 2}{\log 75 - \log 2}. $  Three different
cases of the self-similar measure $\nu , $ which will be determined by the values below, are considered.

\begin{itemize}
\item[(i)]  $s=\frac{1}{3}, $ which implies $D(\nu) =\frac{\log 2}{\log 3} > \kappa , $
\item[(ii)]  $s=\frac{1}{7}, $ which implies $D(\nu) =\frac{\log 2}{\log 7} < \kappa , $ and
\item[(iii)]  $s=\frac{\sqrt{6}}{15}, $ which implies $D(\nu) =\frac{2 \log 2}{\log 75 - \log 2} = \kappa . $
\end{itemize}
In order to provide further insight to the question posed at the end of the article, we will also consider $s=\frac{1}{5}, $ which implies $D(\nu) =\frac{\log 2}{\log 5} > \kappa . $
In each case, we will construct the canonical sequences to investigate the quantization for the associated condensation measures while exhibiting the optimal sets of $n$-means for $n\geq 4 $ and determining the quantization dimensions.

\bigskip

\section{Condensation measure $P$ with self-similar measure $\gn$ satisfying $D(\gn)>\gk$}\label{second}

In this case $s=\frac{1}{3}; $ hence, from the general results obtained in the previous section, we have

\begin{itemize}

\item $E(\nu)=\frac12,\ W=V(\nu)=\frac{1}{200};\ E(P)=\frac12, \ V=V(P)=\frac{65}{584}, $
\item $\alpha_1=\{\frac12 \}, $  with $V_1=V(P), $
\item $\alpha_2 =\{\frac{19}{90}, \frac{71}{90} \} $ with $V_2=\frac{32929}{1182600}, $ and
\item $ \alpha_3 =\{\frac{1}{10}, \frac{1}{2}, \frac{9}{10} \} $ with $V_3=\frac{203}{43800}. $

\end{itemize}

\subsection{Essential lemmas and propositions}


\begin{lemma} \label{lemma0001}  Let $\gb:=\set{c, 1}$, where $0<c<1$. Then,
$\int\mathop{\min}\limits_{a\in\gb}(x-a)^2 dP=\frac{1517}{43800},$
and the minimum occurs when $c=\frac 3{10}$.
\end{lemma}
\begin{proof}
Since $\frac 35<\frac 12(\frac 3{10}+1)=\frac{13}{20}<\frac 45$, the distortion error due to the set $\gb:=\set{\frac 3{10}, 1}$ is
\[\int\min_{a\in\gb}(x-a)^2 dP=\int_{J_1\uu L}(x-\frac 3{10})^2 dP+\int_{J_2}(x-1)^2 dP=\frac{1517}{43800}.\]
Let $\ga:=\set{a, 1}$ be an optimal set of two-means for which the minimum in the hypothesis occurs, and $\tilde V_2$ is the corresponding quantization error. Then, $\tilde V_2\leq  \frac{1517}{43800} . $ Suppose that $a\leq \frac 15$. Then, since $\frac 12(\frac 15+1)=\frac 35$, we have the distortion error as
\[\int_{J_1}(x-S_1(\frac 12))^2 dP+\int_{L}(x-\frac 15)^2 dP+\int_{J_2}(x-1)^2 dP=\frac{1663}{43800}>\tilde V_2,\]
which leads to a contradiction. So, we can assume that $\frac 15<a$. If $a\geq \frac 12$, then
\[\tilde V_2\geq \int_{J_1}(x-\frac 12)^2 dP+\int_{T_1(L)}(x-\frac 12)^2 dP=\frac{65}{1168}>\tilde V_2,\]
which is a contradiction.  Next, if $\frac 25\leq a<\frac 12 , $ then, as $\frac 35<\frac 12(\frac 12+1)=\frac 34<\frac 45$, we have
\[\tilde V_2\geq  \int_{J_1}(x-\frac 25)^2 dP+\int_{T_2(L)}(x-\frac 12)^2 dP+\int_{J_2}(x-1)^2 dP=\frac{3253}{87600}\geq \tilde V_2,\]
which is also a contradiction. So, we can assume that $\frac 15< a< \frac 25$, and then notice that $\frac 12(\frac 15+1)=\frac 35<\frac 12(a+1)<\frac 12(\frac 25+1)<\frac 45$ yielding the fact that
$a=E(X : X \in J_1\uu L)=\frac{1}{P(J_1\uu L)}(P(J_1) S_1(\frac 12)+P(L) \frac 12)=\frac 12(\frac 1{10}+\frac 12)=\frac 3{10}$, and the corresponding quantization error is $\tilde V_2=\frac{1517}{43800}$.
\end{proof}
\begin{cor} \label{cor2} Let $\gb:=\set{c, \frac 15}$, where $0<c<\frac 15$. Then,
$\int_{J_1}\mathop{\min}\limits_{a\in\gb}(x-a)^2 dP=\frac{1517}{3285000},$
and the minimum occurs when $c=\frac 3{50}$.

\end{cor}
\begin{proof} By Lemma~\ref{lemma0001}, we have
\begin{align*}
&\int_{J_1}\min_{a\in \gb} (x-a)^2 dP=\frac 13 \int_{J_1}\min_{a\in \gb} (x-a)^2 d(P\circ S_1^{-1})=\frac 13 \int\min_{a\in \gb} (S_1(x)-a)^2 dP\\
&=\frac 13 \int\min_{a\in S_1^{-1}(\gb)} (S_1(x)-S_1(a))^2 dP=\frac 1{75} \int\min_{a\in S_1^{-1}(\gb)}(x-a)^2 dP=\frac{1517}{3285000},
\end{align*}
which occurs when $c=\frac 3{50}$.
\end{proof}

\begin{lemma}  \label{lemma001} Let $\ga$ be an optimal set of four-means. Then, $\ga\ii J_1\neq \es$, $\ga \ii J_2\neq \es$, and $\ga\ii L\neq \es$. Moreover, $\ga$ does not contain any point from the open intervals $(\frac 15, \frac 25)$ and $(\frac 35, \frac 45)$.
\end{lemma}

\begin{proof} Let us first consider the set $\gb:=\set{S_1(\frac 12), T_1(\frac 12), T_2(\frac 12), S_2(\frac 12)}$. Then,
 \begin{align*}
 \int \min_{a\in \gb}(x-a)^2 dP & =2\Big(\int_{J_{1}}(x-S_{1}(\frac 12))^2 dP+\int_{T_1(L)}(x-T_{1}(\frac 12))^2 dP\Big) \\
 &=2\Big(\frac 1{75} V+\frac 13 \frac 1{18} W\Big)
=\frac{1243}{394200}.
\end{align*}
Since $V_4$ is the quantization error for four-means, we have $V_4\leq \frac{1243}{394200} . $ Let $\ga:=\set{a_1<a_2<a_3<a_4}$ be an optimal set of four-means.  We first show that $\ga\ii J_1\neq \es$. For the sake of contradiction, assume that $\ga\ii J_1=\es$. Then $\frac 15<a_1$, which yields
\[V_4\geq \int_{J_1}(x-\frac 15)^2 dP=\frac{211}{43800}>V_4,\]
which is a contradiction. Thus, we can assume that $\ga\ii J_1\neq \es$. Similarly, we can show that $\ga\ii J_2\neq\es$.
We now show that $\ga \ii L\neq \es$. For the sake of contradiction, assume that $\ga \ii L=\es$. Suppose that $a_2>\frac 35$. Then,
\begin{align*}
V_4&\geq \int_{J_1}(x-a_1)^2 dP+\int_L(x-\frac 35)^2dP\geq \int_{J_1}(x-S_1(\frac 12))^2 dP+\int_L(x-\frac 35)^2dP\\
&=\frac 1{75} V+ \frac 13 (W+(\frac 12-\frac 35)^2)=\frac{71}{10950}>V_4,
\end{align*}
which leads to a contradiction. So, we can assume that $a_2<\frac 25$. Similarly, we have $\frac 35<a_3$. Due to symmetry of $P$, the following two cases can occur:

Case 1. $\frac 13\leq a_2<\frac 25$ and $\frac 35<a_3\leq \frac 23$.

In this case, $\frac 13\leq a_2<\frac 25$ implies $\frac 12(a_1+a_2)<\frac 15$ yielding $a_1<\frac 25-a_2\leq \frac 25-\frac 13=\frac 1{15}<\frac{2}{25}$. Thus, due to symmetry, we have
\begin{align*}
&V_4\geq 2\Big(\int_{S_1(L)\uu J_{12}}(x-\frac 1{15})^2 dP+\int_{T_1(L)}(x-\frac 25)^2 dP\Big)=\frac{110597}{29565000}>V_4,
\end{align*}
which gives a contradiction.

Case 2. $a_2\leq \frac 13$ and $\frac {2}{3}\leq a_3$.

In this case we have, \[V_4\geq 2 \int_{T_1(L)}(x-\frac 13)^2 dP=\frac{19}{5400}>V_4,\ \text{which is a contradiction.} \]

By Case~1 and Case~2, we deduce that $\ga\ii L\neq \es$. We now show that $\ga$ does not contain any point the open intervals $(\frac 15, \frac 25)$ and $(\frac 35, \frac 45)$. Suppose that $a_2\in (\frac 15, \frac 25)$. Then, notice that $a_3\in L$ and $a_4\in J_2$. Again, two cases can arise:

Case~I. $\frac 13\leq a_2<\frac 25$.

In this case, $\frac 12(a_1+a_2)<\frac 15$ implying $a_1<\frac 25-a_2\leq \frac 25-\frac 13=\frac 1{15}<\frac 2{25}$, and so
\begin{align*}
V_4\geq \int_{S_1(L)\uu J_{12}}(x-\frac 1{15})^2 dP+\int_{T_1(L)}(x-\frac 25)^2 dP  +\int_{T_2(L)}(x-T_2(\frac 12))^2 dP+\int_{J_2}(x-S_2(\frac 12))^2 dP;
\end{align*}
and hence, $V_4\geq \frac{101911}{29565000}>V_4$, which is a contradiction.

Case~II. $\frac 15< a_2<\frac 13$.

In this case, $\frac 12(a_2+a_3)>\frac 25$ implying $a_3>\frac 45-a_2\geq \frac 45-\frac 13=\frac 7{15}$. Recall Corollary~\ref{cor2}, and $\int_{J_2}(x-S_2(\frac 12))^2 dP=\frac 1{75}V$.  The following subcases arise:
\begin{itemize}
\item[(i)] If $\frac 7{15}< a_3\leq \frac{22}{45},$ then
\[V_4\geq \frac{1517}{3285000}+\int_{T_1(L)}(x-\frac 7{15})^2 dP  +\int_{T_2(L)}(x-\frac{22}{45})^2 dP+\frac 1{75} V=\frac{294859}{88695000}>V_4, \]
\item[(ii)] If $\frac{22}{45}\leq a_3\leq \frac{1}{2},$ then $T_{111}(\frac 35)<\frac 12(\frac 13+\frac {22}{45})<T_{112}(\frac 25)$ implying that
\begin{align*}
V_4& \geq \frac{1517}{3285000}+\int_{T_{111}(L)}(x-\frac 1{3})^2 dP  +\int_{T_{112}(L)\uu T_{12}(L)}(x-\frac{22}{45})^2 dP+\int_{T_2(L)}(x-\frac 12)^2dP+\frac 1{75} V\\
&=\frac{218863}{66521250}>V_4,
\end{align*}
\item[(iii)] If $\frac 12\leq a_3\leq \frac{91}{180},$
then $T_{111}(\frac 35)<\frac 12(\frac 13+\frac {1}{2})<T_{1122}(\frac 25)$ implying that
\begin{align*}
V_4& \geq \frac{1517}{3285000}+\int_{T_{111}(L)}(x-\frac 1{3})^2 dP  +\int_{T_{1122}(L)\uu T_{12}(L)}(x-\frac{1}{2})^2 dP+\int_{T_2(L)}(x-\frac {91}{180})^2dP+\frac 1{75} V\\
&=\frac{121358413}{38316240000}>V_4,
\end{align*}
\item[(iv)] If $\frac {91}{180}\leq a_3\leq \frac{23}{45},$
then $T_{1121}(\frac 35)<\frac 12(\frac 13+\frac {91}{180})<T_{1122}(\frac 25)$ implying that
\begin{align*}
V_4& \geq \frac{1517}{3285000}+\int_{T_{111}(L)\uu T_{1121}(L)}(x-\frac 1{3})^2 dP  +\int_{T_{1122}(L)\uu T_{12}(L)}(x-\frac{91}{180})^2 dP+\int_{T_2(L)}(x-\frac {23}{45})^2dP\\
&+\frac 1{75} V=\frac{83388217}{25544160000}>V_4,
\end{align*}
\item[(v)] If $ \frac{23}{45}\leq a_3\leq \frac{31}{60},$
then $T_{11}(\frac 35)=\frac 12(\frac 13+\frac {23}{45})$ implying that
\begin{align*}
V_4& \geq \frac{1517}{3285000}+\int_{T_{11}(L)}(x-\frac 1{3})^2 dP  +\int_{T_{12}(L)}(x-\frac{23}{45})^2 dP+\int_{T_2(L)}(x-\frac {31}{60})^2dP\\
&+\frac 1{75} V=\frac{381587}{118260000}>V_4,
\end{align*}
\item[(vi)] If $\frac {31}{60}\leq a_3\leq \frac{21}{40},$
then $T_{11}(\frac 35)<\frac 12(\frac 13+\frac {31}{60})<T_{12}(\frac 25)$ implying
\begin{align*}
V_4& \geq \frac{1517}{3285000}+\int_{T_{11}(L)}(x-\frac 1{3})^2 dP  +\int_{T_{12}(L)}(x-\frac{31}{60})^2 dP+\int_{T_2(L)}(x-\frac {21}{40})^2dP\\
&+\frac 1{75} V=\frac{1491673}{473040000}>V_4,
\end{align*}
\item[(vii)] If $\frac{21}{40}\leq a_3\leq \frac{127}{240},$
then $T_{11}(\frac 35)<\frac 12(\frac 13+\frac {21}{40})<T_{12}(\frac 25)$ implying
\begin{align*}
V_4& \geq \frac{1517}{3285000}+\int_{T_{11}(L)}(x-\frac 1{3})^2 dP  +\int_{T_{12}(L)}(x-\frac{21}{40})^2 dP+\int_{T_2(L)}(x-\frac{127}{240})^2dP\\
&+\frac 1{75} V=\frac{6034217}{1892160000}>V_4,
\end{align*}
\item[(viii)] If $ \frac{127}{240}\leq a_3\leq \frac{8}{15},$
then $T_{11}(\frac 35)<\frac 12(\frac 13+\frac{127}{240})<T_{12}(\frac 25)$ implying
\begin{align*}
V_4& \geq \frac{1517}{3285000}+\int_{T_{11}(L)}(x-\frac 1{3})^2 dP  +\int_{T_{12}(L)}(x-\frac{127}{240})^2 dP+\int_{T_2(L)}(x-\frac{8}{15})^2dP\\
&+\frac 1{75} V=\frac{12070259}{3784320000}>V_4,
\end{align*}
\item[(ix)] If $ \frac{8}{15}\leq a_3\leq \frac 12(\frac 8{15}+T_{212}(\frac 25))=\frac{73}{135},$ then $T_{11}(\frac 35)<\frac 12(\frac 13+\frac{8}{15})<T_{12}(\frac 25)$ implying
\begin{align*}
V_4& \geq \frac{1517}{3285000}+\int_{T_{11}(L)}(x-\frac 1{3})^2 dP  +\int_{T_{12}(L)}(x-\frac{8}{15})^2 dP+\int_{T_{22}(L)}(x-\frac{73}{135})^2dP\\
&+\frac 1{75} V=\frac{10098449}{3193020000}>V_4,
\end{align*}
\item[(x)] If $\frac 12(\frac 8{15}+T_{212}(\frac 25))=\frac{73}{135}\leq a_3\leq \frac 5 9=T_{21}(\frac 35),$ then $T_{11}(\frac 35)<\frac 12(\frac 13+\frac{73}{135})<T_{12}(\frac 25)$ implying
\begin{align*}
V_4& \geq \frac{1517}{3285000}+\int_{T_{11}(L)}(x-\frac 1{3})^2 dP  +\int_{T_{12}(L)}(x-\frac{73}{135})^2 dP+\int_{T_{22}(L)}(x-\frac{5}{9})^2dP\\
&+\frac 1{75} V=\frac{10098449}{3193020000}>V_4,
\end{align*}
\item[(xi)] If $ T_{21}(\frac 35)=\frac 59\leq a_3,$ then
$\frac 12(\frac 13+\frac{5}{9})=T_{12}(\frac 25)$ implying that
$$ V_4 \geq \frac{1517}{3285000}+\int_{T_{11}(L)}(x-\frac 1{3})^2 dP  +\int_{T_{12}(L)}(x-\frac 59)^2 dP+\frac 1{75} V
=\frac{584243}{177390000}>V_4,$$
\end{itemize}
each of which is a contradiction.  Thus, by Case~I and Case~II, it follows that  $a_2\not \in (\frac 15, \frac 25)$, i.e., $\ga$ does not contain any point from the open interval  $(\frac 15, \frac 25)$. Reflecting the situation with respect to the point $\frac 12$, we  also deduce that $\ga$ does not contain any point from the open interval  $(\frac 35, \frac 45)$. Thus, the proof of the lemma is complete.
\end{proof}

\begin{prop} \label{prop3}
$\set{S_1(\frac 12), T_1(\frac 12), T_2(\frac 12), S_2(\frac 12)}$ is an optimal set of four-means for $P$ with quantization error $V_4=\frac{1243}{394200}.$
\end{prop}

\begin{proof} If $\gb:=\set{S_1(\frac 12), T_1(\frac 12), T_2(\frac 12), S_2(\frac 12)} , $ then
 \begin{align*}
 \int \min_{a\in \gb}(x-a)^2 dP &=2\Big(\int_{J_{1}}(x-S_{1}(\frac 12))^2 dP+\int_{T_1(L)}(x-T_{1}(\frac 12))^2 dP\Big) \\
 &=2\Big(\frac 1{75} V+\frac 13 \frac 1{18} W\Big)
=\frac{1243}{394200}.
\end{align*}
Since $V_4$ is the quantization error for four-means, $V_4\leq
\frac{1243}{394200} . $ Let $\ga:=\set{a_1, a_2, a_3, a_4}$ be an optimal set of four-means. Since optimal quantizers are the centroids of their own Voronoi regions, without any loss of generality, we can assume that $0<a_1<a_2<a_3<a_4<1$. By Lemma~\ref{lemma001}, we see that $\ga$ contains points from $J_1$, $L$, $J_2$, and $\ga$ does not contain any point from the open intervals $(\frac 15, \frac 25)$, $(\frac 35, \frac 45)$. We now show that $\te{card}(\ga\ii L)=2$. Suppose that $\te{card}(\ga\ii L)=1$. Then, without any loss of generality, we can assume that
$\te{card}(\ga\ii J_1)=2$ and $\te{card}(\ga\ii J_2)=1$. Recall Proposition~\ref{prop00021}. Then, by Lemma~\ref{lemma4}, we have
$\ga\ii J_1=\set{S_1(\frac {19}{90}), S_1(\frac {71}{90})}$, $\ga\ii L=\set{\frac 12}$, and $\ga\ii J_2=\set{S_2(\frac 12)}$. Then, we see that the quantization error is $\frac{312379}{88695000}>V_4$, which is a contradiction. Hence, $\te{card}(\ga\ii L)=2$, $\te{card}(\ga\ii J_1)=\te{card}(\ga\ii J_2)=1$ implying the fact that $\ga=\set{S_1(\frac 12), T_1(\frac 12), T_2(\frac 12), S_2(\frac 12)}$ is an optimal set of four-means with quantization error $V_4=\frac{1243}{394200} . $
\end{proof}

\begin{lemma}  \label{lemma002} Let $\ga$ be an optimal set of five-means. Then, $\ga\ii J_1\neq \es$, $\ga \ii J_2\neq \es$, and $\ga\ii L\neq \es$. Moreover, $\ga$ does not contain any point from the open intervals $(\frac 15, \frac 25)$ and $(\frac 35, \frac 45)$.
\end{lemma}

\begin{proof} First, consider the set $\gb:=\set{S_1(\frac {19}{90}), S_1(\frac{71}{90}), T_1(\frac 12), T_2(\frac 12), S_2(\frac 12)}$. Then,
\begin{align*}
\int \min_{a\in \gb}(x-a)^2 dP
&= 2 \Big(\int_{J_{11}\uu S_1(T_1(L))}(x-S_1(\frac {19}{90}))^2 dP+\int_{T_1(L)}(x-T_{1}(\frac 12))^2 dP\Big) \\
&+\int_{J_2}(x-S_2(\frac 12))^2 dP
=2\Big(\frac 1{75} \frac 12 V_2+\frac 13 \frac 1{18} W\Big)+\frac 1{75} V=\frac{180979}{88695000}.
\end{align*}
Since $V_5$ is the quantization error for five-means, we have $V_5\leq \frac{180979}{88695000} .$ Let $\ga:=\set{a_1<a_2<a_3<a_4<a_5}$ be an optimal set of five-means. Proceeding in the similar way as shown in the proof of Proposition~\ref{prop2}, we have $0<a_1<\frac 15$ and $\frac 45<a_5<1$ implying $\ga \ii J_1\neq\es$ and $\ga \ii J_2\neq \es$. We now show that $\ga \ii L\neq \es$. For the sake of contradiction, assume that $\ga \ii L=\es$. \\
\noindent If $\frac 35<a_2$. Then,
\[V_5\geq \int_{J_1}(x-S_1(\frac 12))^2 dP+\int_{L}(x-\frac 35)^2 dP=\frac 1{75} V+\frac 13(W+(\frac 12-\frac 35)^2)=\frac{71}{10950}>V_5. \] \\
\noindent If $a_2<\frac 25<\frac 35<a_3$, the following three cases can arise:
\begin{itemize}
\item[(i)] $\frac 3{10}\leq a_2<\frac 25$ and $\frac 35<a_3\leq \frac {7}{10}$. Then, $\frac 12 (a_1+a_2)<\frac 15$ and $\frac 45<\frac 12(a_3+a_4)$ implying $a_1<\frac 2 5-a_2\leq \frac  1{10}$ and $a_4>\frac 85-a_3 \geq \frac 9{10}$. Hence, by symmetry,
\begin{align*}
&V_5\geq \int_{J_{11}\uu S_1(T_1(L))}(x-S_1(\frac {19}{90}))^2 dP+ 2 \int_{S_1(T_2(L))\uu J_{12}}(x-\frac 1{10})^2 dP+2 \int_{T_1(L)}(x-\frac 25)^2 dP\\
&=\frac{394729}{177390000}>V_5.
\end{align*}
\item[(ii)] $\frac 3{10}\leq a_2<\frac 25$  and $\frac {7}{10}\leq a_3$.
Then, $T_{21}(\frac 25)<\frac 12(\frac 25+\frac 7{10})<T_{21}(\frac 35)$, and so
\begin{align*}
V_5&\geq \int_{J_{11}\uu S_1(T_1(L))}(x-S_1(\frac {19}{90}))^2 dP+  \int_{S_1(T_2(L))\uu J_{12}}(x-\frac 1{10})^2 dP+ \int_{T_1(L)}(x-\frac 25)^2 dP\\
&+\int_{T_{22}(L)}(x-\frac 7{10})^2 dP=\frac{794483}{354780000}>V_5.
\end{align*}
\item[(iii)] $a_2\leq \frac 3{10}$ and $\frac {7}{10}\leq a_3$.
Then, due to symmetry,
\begin{align*} &V_5\geq 2 \int_{T_1(L)}(x-\frac 3{10})^2 dP=\frac{11}{1800}>V_5.
\end{align*}
\end{itemize}
In each of the cases we reach a contradiction; hence, we can assume that $\ga \ii L\neq \es$.

Next, recall Proposition~\ref{prop00021}. If $\te{card}(\ga\ii L)=3$, then
\[V_5\geq \int_{J_1}(x-S_1(\frac 12))^2 dP+\int_{J_2}(x-S_2(\frac 12))^2 dP=\frac 2 {75}V=\frac{13}{4380} >V_5,\]
which gives a contradiction. So, we can assume that $1\leq \te{card}(\ga\ii L)\leq 2$.  If $\te{card}(\ga\ii L)=1 , $ then, due to symmetry, the following two cases can arise:

Case~1. $a_3=\frac 12$, $\frac 3{10}\leq a_2<\frac 25$ and $\frac 35<a_4\leq \frac {7}{10}$:
As shown above, we have $a_1<\frac 1{10}$ and $a_5>\frac 9{10}$. Moreover, $T_{1212}(\frac 25)<\frac 12(\frac 25+\frac 12)<T_{1212}(\frac 35)$. Thus,
\begin{align*}
V_5&\geq 2\int_{J_{11}\uu S_1(T_1(L))}(x-S_1(\frac {19}{90}))^2 dP+2 \int_{S_1(T_2(L))\uu J_{12}}(x-\frac 1{10})^2 dP+2\int_{T_{11}(L)\uu T_{1211}}(x-\frac 25)^2 \\
&  +2\int_{T_{122}(L)}(x-\frac 12)^2=\frac{40002139}{19158120000}>V_5,\ \text{contradiction.}
\end{align*}

Case~II. $a_3=\frac 12$, $a_2\leq \frac 3{10}$ and $\frac {7}{10}\leq a_4$:
First, notice that $\frac 12(\frac 3{10}+\frac 12)=\frac 25$ and $\frac 12(\frac 12+\frac 7{10})=\frac 35$ implying the fact that the Voronoi regions of $a_2$ and $a_4$ do not contain any point from $L$, and thus, we have $a_1, a_2\in J_1$, and $a_4, a_5\in J_2$. Hence,
\[V_5\geq \int_{J_1}\min_{a\in \set{a_1, a_2}}(x-a)^2 dP+\int_L(x-\frac 12)^2+\int_{J_2}\min_{a\in \set{a_4, a_5}}(x-a)^2 dP=\frac 2{75} \frac 12 V_2+\frac 13 W+\frac 2{75} \frac 12 V_2\]
yielding
$V_5\geq \frac{213683}{88695000}>V_5$, which gives a contradiction.

Therefore, we can assume that $\te{card}(\ga\ii L)=2$. We now show that $\ga$ does not contain any point from the open intervals $(\frac 15, \frac 25)$ and $(\frac 35, \frac 45)$. Suppose that $\ga$ contains a point from $(\frac 15, \frac 25)$. In that case, $\ga$ does not contain any point from $(\frac 35, \frac 45)$ as $a_1\in J_1$ and $a_3, a_4\in L$, and $a_5 \in J_2$; consequently, the following two possibilities arise:\\
1.  $\frac 3{10}\leq a_2<\frac 25$: Then, $a_1< \frac 1{10},$ and
\begin{align*}
&V_5\geq \int_{J_{11}\uu S_1(T_1(L))}(x-S_1(\frac {19}{90}))^2 dP+ \int_{S_1(T_2(L))\uu J_{12}}(x-\frac 1{10})^2 dP+ \int_{T_1(L)}(x-\frac 25)^2 dP\\
& \qquad \qquad \qquad +\int_{T_{21}(L)}(x-T_{21}(\frac 12))^2 dP+\int_{T_{22}(L)}(x-T_{22}(\frac 12))^2 dP+\int_{J_2}(x-S_2(\frac 12))^2 dP\\
&= \frac{59863}{22173750}>V_5,
\end{align*} \\
2. $\frac 15<a_2\leq \frac 3{10}$:
Then, $\frac 12(a_2+a_3)>\frac 25$ implying $a_3>\frac 45-a_2\geq \frac 45-\frac 3{10}=\frac 12$, and so
\[V_5\geq  \int_{T_1(L)}(x-\frac 12)^2 dP+\int_{T_{21}(L)}(x-T_{21}(\frac 12))^2 dP+\int_{T_{22}(L)}(x-T_{22}(\frac 12))^2 dP+\int_{J_2}(x-S_2(\frac 12))^2 dP\]
yielding $V_5\geq \frac{4129}{1773900}>V_5$.

Again, in both cases we reach a contradiction.
Thus, it follows that $\ga$ does not contain any point from the open interval $(\frac 15, \frac 25)$. Reflecting the situation with respect to the point $\frac 12$, we can also show that $\ga$ does not contain any point from the open interval $(\frac 35, \frac 45)$.
\end{proof}

\begin{lemma} \label{lemma003}
Let $\ga$ be an optimal set of $n$-means for $n=6$. Then, $\ga\ii J_1\neq \es$, $\ga \ii J_2\neq \es$ and $\ga\ii L\neq \es$. Moreover, $\ga$ does not contain any point from the open intervals $(\frac 15, \frac 25)$ and $(\frac 35, \frac 45)$.
\end{lemma}

\begin{proof}  Let $\ga:=\set{a_1, a_2, a_3, a_4, a_5, a_6}$ be an optimal set of six-means.  Since optimal quantizers are the centroids of their own Voronoi regions, without any loss of generality, we can assume that $0<a_1<a_2<a_3<a_4<a_5<a_6<1$.
Now, consider the set of points $\gb:=\set{S_1(\frac {19}{90}), S_1(\frac{71}{90}), T_1(\frac 12), T_2(\frac 12), S_2(\frac {19}{90}), S_2(\frac{71}{90})}$. Then,
 \begin{align*}
& \int \min_{a\in \gb}(x-a)^2 dP\\
&= 2 \Big(\int_{J_{11}\uu S_1(T_1(L))}(x-S_1(\frac {19}{90}))^2 dP+\int_{S_1(T_2(L))\uu J_{12}}(x-S_1(\frac {71}{90}))^2 dP+\int_{T_1(L)}(x-T_{1}(\frac 12))^2 dP\Big)\\
&=2\Big(\frac 1{75} \frac 12 V_2+\frac 1{75} \frac 12 V_2+\frac 13 \frac 1{18} W\Big)=\frac{82283}{88695000}.
\end{align*}
Since $V_6 $ is the quantization error for six-means, we have $V_6\leq \frac{82283}{88695000}. $ Proceeding in the similar way as in the proof of Proposition~\ref{prop2}, we have $0<a_1<\frac 15$ and $\frac 45<a_4<1$ implying $\ga \ii J_1\neq\es$ and $\ga \ii J_2\neq \es$. $P$ being symmetric about $\frac 12$, we can assume that there are three optimal quantizers to the left of $\frac 12$ and three optimal quantizers to the right of $\frac 12$. We now show that $\ga \ii L\neq \es$. Suppose that $\ga \ii L= \es$. Then, $a_3<\frac 25$ and $\frac 35< a_4$. Due to symmetry, first suppose that $\frac 3{10}\leq a_3<\frac 25$ and $\frac 35<a_4\leq \frac {7}{10}$. Then, $\frac 12 (a_2+a_3)<\frac 15$ and $\frac 45<\frac 12(a_4+a_5)$ implying $a_2<\frac 2 5-a_3\leq\frac 25-\frac 3{10}=\frac  1{10}$ and $a_5>\frac 85-a_4 \geq \frac 85-\frac 7{10}=\frac 9{10}$. Hence, by Corollary~\ref{cor1}, we have
\begin{align*} \label{eq0000}
V_6\geq 2\Big(\int_{S_1(T_2(L))\uu S_1(J_2)}(x-\frac 1{10})^2 dP\Big)=\frac{13}{8760}>V_6,
\end{align*}
which gives a contradiction. Next, suppose that $a_3\leq \frac 3{10}$ and $\frac 7{10}\leq a_4$. Then, \[V_6\geq 2 \int_{T_1(L)}(x-\frac 3{10})^2 dP=\frac{11}{1800}>V_6,\]
which leads to another contradiction. So, we can assume that $\ga\ii L\neq \es$. We now show that $\ga$ does not contain any point from the open intervals $(\frac 15, \frac 25)$ and $(\frac 35, \frac 45)$. Suppose that $\ga$ contains a point from the open interval $(\frac 15, \frac 25)$. Then, due to symmetry, $\ga$ will also contain a point from $(\frac 35, \frac 45)$, and the only possible case is that $a_2 \in (\frac 15, \frac 25)$ and $a_5 \in (\frac 35, \frac 45)$. The following two case can happen:

Case~1. $\frac 3{10}\leq a_2<\frac 25$ and $\frac 35<a_5\leq \frac {7}{10}$: Then,
\begin{align*}
V_6\geq 2\Big(\int_{S_1(T_2(L))\uu S_1(J_2)}(x-\frac 1{10})^2 dP\Big)=\frac{13}{8760}>V_6,\ \text{a contradiction.}
\end{align*}

Case~2. $\frac 15<a_2\leq \frac 3{10}$ and $\frac 7{10}\leq a_5<\frac 45$:
Then, $\frac 12(a_2+a_3)>\frac 25$ implies $a_3>\frac 45-a_2\geq \frac 45-\frac 3{10}=\frac 12$, i.e., $a_3>\frac 12$. Similarly, $a_4<\frac 12$, which is a contradiction, as $a_3<a_4$.

Hence, $\ga$ does not contain any point from the open intervals $(\frac 15, \frac 25)$ and $(\frac 35, \frac 45)$. 
\end{proof}

\begin{prop}  \label{prop003} Let $\ga_n$ be an optimal set of $n$-means for all $n\geq 3$. Then, $\ga_n\ii J_1\neq \es$, $\ga_n \ii J_2\neq \es$, and $\ga_n\ii L\neq \es$. Moreover, $\ga_n$ does not contain any point from the open intervals $(\frac 15, \frac 25)$ and $(\frac 35, \frac 45)$.
\end{prop}

\begin{proof}By Proposition~\ref{prop2}, Lemma~\ref{lemma001}, Lemma~\ref{lemma002}, and Lemma~\ref{lemma003}, it follows that the proposition is true for $3\leq n\leq 6$. We now show that the proposition is true for all $n\geq 7$. Let $\ga_n:=\set{a_1, a_2, a_3, \cdots, a_n}$ be an optimal set of $n$-means for all $n\geq 7$. Since optimal quantizers are the centroids of their own Voronoi regions, without any loss of generality, we can assume that $0<a_1<a_2<a_3<\cdots <a_n<1$.
Consider the set of seven points \[\gb_7=\set{S_1(\frac {19}{90}), S_1(\frac {71}{90}), T_1(\frac 12), T_2(\frac 12), S_{21}(\frac 12), S_2(\frac 12), S_{22}(\frac12)}.\]
Then, $\int\mathop{\min}\limits_{a\in \gb_7}(x-a)^2 dP=\frac{10967}{17739000}. $
Since $V_n$ is the quantization error for $n$-means for all $n\geq 7$, we have $V_n\leq V_7\leq \frac{10967}{17739000}. $ Proceeding in the similar way as in the proof of Proposition~\ref{prop2}, we have $0<a_1<\frac 15$ and $\frac 45<a_n<1$  implying $\ga_n \ii J_1\neq\es$ and $\ga_n \ii J_2\neq \es$. We now show that $\ga_n \ii L\neq \es$. Suppose that $\ga_n \ii L= \es$. Let $j=\max\set{i : a_i< \frac 25 \te{ for all } 1\leq i\leq n}$, and then $a_j<\frac 25$. First, assume that $\frac 3{10}\leq a_j<\frac 25$.
Then, $\frac 12(a_{j-1}+a_j)<\frac 15$ implying $a_{j-1}<\frac 25-a_j\leq \frac 25-\frac 3{10}=\frac 1{10}$, and so
\begin{align*}
V_n\geq \int_{S_1(T_2(L))\uu S_1(J_2)}(x-\frac 1{10})^2 dP+\int_{T_1(L)}(x-\frac 25)^2dP+\int_{T_2(L)}(x-\frac 35)^2dP=0.00129756>V_n,
\end{align*}
which is a contradiction. Next, assume that $a_j\leq \frac 3{10}$. Then, notice that $T_{11}(\frac 35)=\frac{19}{45}<\frac 12(\frac 3{10}+\frac 35)=\frac 9{20};$ hence,
\begin{align*}
V_n&\geq \int_{T_{11}(L)}(x-\frac 3{10})^2 dP=\frac{67}{64800}>V_n,
\end{align*}
which leads to a contradiction. Thus, we can conclude that $\ga_n\ii L\neq \es$.
We now prove that $\ga_n$ does not contain any point from the open intervals $(\frac 15, \frac 25)$ and $(\frac 35, \frac 45)$. Suppose that $\ga_n$ contains a point from the open interval $(\frac 15, \frac 25)$. Let $k=\max\set {i : a_i<\frac 25 \te{ for all } 1\leq i\leq n}$, and then $a_k<\frac 25$.

Case~I. $\frac 3{10}\leq a_k<\frac 15$: Then, $\frac 12(a_{k-1}+a_{k})\leq \frac 15$ implies  $a_{k-1}\leq \frac 1{10}$ which yields
\[V_n\geq \int_{S_1(T_2(L))\uu S_1(J_2)}(x-\frac 1{10})^2 dP=\frac{13}{17520}>V_n,\ \text{which is a contradiction.} \]

Case~II. $\frac 15<a_k\leq \frac 3{10}$: Then, $\frac 12(a_k+a_{k+1})> \frac 25$ implies $a_{k+1}> \frac 12$ which yields
\[V_n\geq \int_{T_1(L)}(x-\frac 12)^2dP=\frac{1}{1200}>V_n,\ \text{which also leads to a contradiction.} \]

By Case~I and Case~II, we can assume that $\ga_n$ does not contain any point from the open interval $(\frac 15, \frac 25)$. Reflecting the situation with respect to $\frac 12$, we can also assume that $\ga_n$ does not contain any point from the open interval $(\frac 35, \frac 45)$.
\end{proof}

The following proposition gives a property of the $n$-th quantization error.

\begin{prop} \label{prop55}
Let $\ga$ be an optimal set of $n$-means for $P$, where $n\geq 3$. Set $\ga_1:=\ga\ii J_1$, $\ga_2:=\ga\ii J_2$ and $\ga_L:=\ga\ii L$. Let
 $n_1=\te{card}(\ga\ii J_1)$, $n_2=\te{card}(\ga\ii J_2)$ and $n_L=\te{card}(\ga\ii L)$. Then,
 \[V_n(P)=\frac 1{75} (V_{n_1}(P)+V_{n_2}(P))+\frac 13 V_{n_L}(\gn).\]
\end{prop}
\begin{proof} By Proposition~\ref{prop003}, $\ga$ does not contain any point from the open intervals $(\frac 15, \frac 25)$ and $(\frac 35, \frac 45)$, and so $n=n_1+n_2+n_L$. By Lemma~\ref{lemma4}, we see that $S_1^{-1}(\ga_1)$ is an optimal set of $n_1$-means for $P$, and $S_2^{-1}(\ga_2)$ is an optimal set of $n_2$-means for $P$. In the similar way it can be seen that $\ga_L$ is an optimal set of $n_L$-means for $P|_L=\gn$, where $P|_L$ is the conditional probability measure of $P$ on $L$ as defined by \eqref{eq345}. Again, $P$-almost surely, the Voronoi region of any point in $\ga\ii J_i$ for $i=1, 2$, does not contain any point from the Voronoi region of any point in $\ga\ii L$, and vice versa. Thus,
\begin{align*}
V_n(P)&=\int_{J_1}\min_{a\in \ga_1}(x-a)^2 dP+\int_{L}\min_{a\in \ga_L}(x-a)^2 dP+\int_{J_2}\min_{a\in \ga_2}(x-a)^2 dP\\
&=\frac 1{75} \int\min_{a\in S_1^{-1}(\ga_1)}(x-a)^2 dP+\frac 13\int\min_{a\in \ga_L}(x-a)^2 dP+\frac 1{75} \int\min_{a\in S_2^{-1}(\ga_2)}(x-a)^2 dP\\
&=\frac 1{75} (V_{n_1}(P)+V_{n_2}(P))+\frac 13 V_{n_L}(\gn),
 \end{align*}
which completes the proof.
\end{proof}

\subsection{Canonical sequences and optimal quantization}

In this section, we first define the \tit{canonical sequences} $\set{a(n)}_{n\geq 1}$ and  $\set{F(n)}_{n\geq 1}$ associated with the condensation system, then calculate the optimal sets of $F(n)$-means, $F(n)$-th quantization error, quantization dimension and quantization coefficients.
\medskip

\noindent{\bf Observation.} Let $A$ be a Borel subset of $\D R$, $\ga_n $ be an optimal set of $n$-means for $P , $ and let $V (P; \ga_n, A)$ be  the quantization error contributed by $\ga_n$ on the set $A$ with respect to $P . $  Using Proposition~\ref{prop00021}, Proposition~\ref{prop003}, and Proposition~\ref{prop55}, it follows that if
\[V(P; \ga_n, J_i)\geq \max\set{V(P; \ga_n, J_j), V(P; L)},\]
for $1\leq i\neq j\leq 2$, then
$\te{card}(\ga_{n+1}\ii J_i)=\te{card}(\ga_n\ii J_i)+1, \, \te{card}(\ga_{n+1}\ii L)=\te{card}(\ga_n\ii L), \te{ and } \te{card}(\ga_{n+1}\ii J_j)=\te{card}(\ga_n\ii J_j);\ \text{and if}$
\[V(P; \ga_n, L)\geq \max\set{V(P; \ga_n, J_i) : 1\leq i\leq 2},\]
then
$\te{card}(\ga_{n+1}\ii J_i)=\te{card}(\ga_n\ii J_i) \te{ for } i=1, 2, \te{ and } \te{card}(\ga_{n+1}\ii L)=\te{card}(\ga_n\ii L)+1.$

In order to avoid routine technicalities, we omit the justification of the Observation.  Notice that the Observation enables one to construct optimal sets $ \ga_n ,\ n\geq 3, $ beginning with the optimal set of three-means.  For example, by Proposition~\ref{prop2}, the set $\ga_3:=\set{S_1(\frac 12), \frac 12, S_2(\frac 12)}$ is an optimal set of three-means which contains one element from each of $J_1$, $L$, and $J_2$. Since $V(P; \ga_3, L)=\frac 13 W>\frac 1{75}V=V(P: \ga_3, J_1)=V(P; \ga_3, J_2)$, by the Observation above, an optimal set $\ga_4$ of four-means must contain two elements from $L$, and one element from each of $J_1$ and $J_2$ yielding $\ga_4=\set{S_1(\frac 12), T_1(\frac 12), T_2(\frac 12), S_2(\frac 12)}$, which is justified by Proposition~\ref{prop3}. Once $\ga_4$ is known, similarly, one calculates an optimal set $\ga_5$ of five-means, and so forth.

Having optimal sets of $n$-means for $1\leq n \leq 6 , $ and utilizing the Observation, we define the canonical sequences that are useful in calculating the optimal sets of $n$-means and the $n$th quantization errors for all $n\in \D N$.


\begin{defi} \label{defi3}  The sequences $\set{a(n)}_{n\geq 1}$ and $\set{F(n)}_{n\geq 1}, $ where
$$ \aligned & a(n)=\frac{1}{4} \left(6 n+(-1)^{n+1}-7\right), \ \text{and}\\
& F(n)=\left\{\begin{array} {ll}
2^n(n+1) & \te{ if } 1\leq n\leq 4,\\
5\cdot 2^n+2^{n-\frac 74} \sum_{k=5}^n 2^{\frac k 2+\frac{(-1)^{k+1}}{4}} & \text{ if } n\geq 5,
\end{array}\right. \endaligned $$
are the canonical sequences associated with the condensation system $(\mathcal{S}, \bold{p}, \nu). $
\end{defi}

\begin{remark} First several terms of the sequences $\{a(n)\}$ and $\{F(n)\}$ are
$$ \aligned & \set{a(n)}_{n\geq 1}=\{0,1,3,4,6,7,9,10,12,13,15,16,18,19,21, \cdots\},  \\
& \set{F(n)}_{n\geq 1}=\{4, 12, 32, 80, 224, 576,1664,4352,12800, 33792, \cdots \}. \endaligned $$
Also, it is easy to see that $a(2k+1)=3k$ and $a(2k)=3k-2$ for all $k\geq 1$.
\end{remark}

\begin{lemma} \label{lemma56}
For any positive integer $k\geq 1$, we have
\begin{itemize}

\item[(i)] $F(2k+1)=2^{2k+1}(1+3 \cdot 2^{k-1})=2\cdot 2^{2k}+3 \cdot 2^{3k}$, and
\item[(ii)] $F(2k)=2^{2k}(1+2^k)=2^{2k}+2^{3k}$.
\end{itemize}
In fact, for any $n\geq 3$, we have $F(n)=2^{a(n)}+2F(n-1)$, and if $n=2$, then $F(2)=2^2+2F(1)$.
\end{lemma}

\begin{proof} Notice that if $k=1$ and $k=2$, then the relations given by  $F(2k+1)$ and $F(2k)$ are clearly true. Let $n=2k+1$ for some positive integer $k\geq 2$. Then,
\begin{align*}
&F(2k+1)=5\cdot 2^{2k+1}+2^{2k+1-\frac 74}\sum_{\ell=5}^{2k+1}2^{\frac \ell 2+\frac{(-1)^{\ell+1}}{4}}\\
&=5\cdot 2^{2k+1}+2^{2k-\frac 34}\Big((2^{\frac 52+\frac 14}+  2^{\frac 72+\frac 14}+\cdots+2^{\frac {2k+1}{2}+\frac 14}) +(2^{\frac 62-\frac 14}
+  2^{\frac 82-\frac 14}+\cdots+2^{\frac {2k}{2}-\frac 14}) \Big)  \\
&=5\cdot 2^{2k+1}+2^{2k-\frac 34}\Big(2^{\frac {11}{4}}(1+2+2^2+\cdots +2^{k-2})+ 2^{\frac {11}{4}}(1+2+2^2+\cdots +2^{k-3})\Big)  \\
&=5\cdot 2^{2k+1}+2^{2k-\frac 34}2^{\frac {11}{4}}(2^{k-1}-1+ 2^{k-2}-1)  =5\cdot 2^{2k+1}+2^{2k+2}(3\cdot 2^{k-2}-2)\\
&=2^{2k+1}(1+3\cdot 2^{k-1}).
\end{align*}
Let $n=2k$ for some positive integer $k\geq 3$. Then,
\begin{align*}
&F(2k)=5\cdot 2^{2k}+2^{2k-\frac 74}\sum_{\ell=5}^{2k}2^{\frac \ell 2+\frac{(-1)^{\ell+1}}{4}}\\
&=\frac 12\Big(5\cdot 2^{2k+1}+2^{2k+1-\frac 74}\sum_{\ell=5}^{2k+1}2^{\frac \ell 2+\frac{(-1)^{\ell+1}}{4}} -2^{2k+1-\frac 74}\cdot 2^{\frac{2k+1}{2}+\frac 14}\Big)=\frac 12 F(2k+1)-2^{3k-1}\\
&=2^{2k}(1+3\cdot 2^{k-1})-2^{3k-1}=2^{2k}(1+2^k).
\end{align*}
To prove $F(n)=2^{a(n)}+2F(n-1)$ for $n\geq 3$, we proceed as follows: Let $n=2k+1$ for some $k\geq 1$. Then,
$F(n)=F(2k+1)=2\cdot 2^{2k}+3 \cdot 2^{3k}=2^{3k}+2(2^{2k}+2^{3k})=2^{a(2k+1)}+2 F(2k)=2^{a(n)}+2F(n-1)$. Let $n=2k$ for some $k\geq 2$. Then, $F(n)=F(2k)=2^{2k}+ 2^{3k}=2^{3k-2}+2^{2k}+3 \cdot 2^{3k-2}=2^{a(2k)}+2F(2k-1)=2^{a(n)}+2F(n-1)$. If $n=2$, then $F(2)=2^2+2F(1)$ directly follows from the definition.
\end{proof}

For $n\in \D N$, we will denote the sequence of sets
$ \ga_{2^{a(n)}}(\gn), \, \uu_{\go \in I} S_\go(\ga_{2^{a(n-1)}}(\gn)), \, \\ \uu_{\go \in I^2} S_\go(\ga_{2^{a(n-2)}}(\gn)), \,  \cdots, \, \uu_{\go \in I^{n-4}} S_\go(\ga_{2^{a(4)}}(\gn)), \, \uu_{ \go \in I^{n-3}} S_\go(\ga_{2^{a(3)}}(\gn)), \, \uu_{\go \in I^{n-2}} S_\go(\ga_{2^2}(\gn)), \, \\ \uu_{\go \in I^{n-1}} S_\go(\ga_{2}(\gn)), \, \set{S_\go(\frac 12) : \go \in I^{n}}$ by $S(a(n))$, $S(a(n-1))$, $S(a(n-2))$,  $\cdots$, $S(a(4))$, $S(a(3))$, $S(2)$, $S(1)$, and $S(0)$, respectively.

Also, for $1\leq \ell\leq 2$, write
\[S^{(2)}(\ell):=\uu_{\go \in I^{n-\ell}} S_\go(\ga_{2^{\ell+1}}(\gn)) \te{ and } S^{(2)(2)}(\ell):=\uu_{\go \in I^{n-\ell}} S_\go(\ga_{2^{\ell+2}}(\gn)), \]  and for $3\leq \ell\leq n$, write
\[S^{(2)}(a(\ell)):=\uu_{\go \in I^{n-\ell}} S_\go(\ga_{2^{a(\ell)+1}}(\gn)) \te{ and } S^{(2)(2)}(a(\ell)):=\uu_{\go \in I^{n-\ell}} S_\go(\ga_{2^{a(\ell)+2}}(\gn)). \]
Further, we write
$S^{(2)}(0):=\set{S_\go(\ga_2(P)) : \go \in I^{n}}=\set{S_\go(\frac {19}{90}), S_\go(\frac{71}{90}) : \go \in I^n}$,
and $S^{(2)(2)}(0)=\uu_{\go \in I^{n}} S_\go(\ga_{2}(\gn))\uu \set{S_\go(\frac 12) : \go \in I^{n+1}}.$ Moreover, for any $\ell\in \D N\uu\set{0}$, if $A:=S(i)$, we identify $S^{(2)}(i)$ and $S^{(2)(2)}(i)$, respectively, by $A^{(2)}$ and $A^{(2)(2)}$; and if $A:=S(a(i))$, we identify $S^{(2)}(a(i))$ and $S^{(2)(2)}(a(i))$, respectively, by $A^{(2)}$ and $A^{(2)(2)}$. Set
\[\ga_{F(n)}:=\ga_{F(n)}(P)=\left\{\begin{array} {ll}
S(1)\uu S(0)  & \te{ if } n=1,\\
S(2)\uu S(1)\uu S(0) & \te{ if } n=2,\\
\left(\mathop{\uu}\limits_{\ell=3}^n S(a(\ell))\right)\uu S(2)\uu S(1)\uu S(0) & \text{ if } n\geq 3.
\end{array}\right.\]
For any real number $x$, let $\lfloor x\rfloor$ denote the greatest integer not exceeding $x$.
For $n\in \D N$, $n\geq 1$, set
\begin{align*} SF(n):& =\set{S(0), S(1), S(2), S(a(3)), S(a(4)), \cdots, S(a(n))}, \\
SF^{(1)}(n):&=\set{S(0), S(4), S(6), \cdots, S(a(2\lfloor \frac n 2\rfloor))},\\
SF^{(2)}(n):&=SF(n)\setminus SF^{(1)}(n).
\end{align*}
In addition, write
\begin{align} \label{eq35} SF^\ast(n): =\set{S(0), S(1), S(2), S(a(3)), & S(a(4)), \cdots, S(a(n)), S^{(2)}(0), S^{(2)}(a(4)), \\
 & S^{(2)} (a(6)), \cdots, S^{(2)}(a(2\lfloor \frac n 2\rfloor))}\notag.
\end{align}
For any element $a\in A \in SF^\ast(n)$, by the Voronoi region of $a$ it is meant the Voronoi region of $a$ with respect to the set $\uu_{B\in SF^\ast(n)} B$.
Similarly, for any $a\in A\in SF(n)$, by the Voronoi region of $a$ it is meant the Voronoi region of $a$ with respect to the set $\uu_{B\in SF(n)} B$. Notice that if $a, b\in A$, where $A \in SF(n)$ or $A\in SF^\ast(n)$, the error contributed by $a$ in the Voronoi region of $a$ is equal to the error contributed by $b$ in the Voronoi region of $b$.  Let us now define an order $\succ$ on the set $SF^\ast(n)$ as follows: For $A, B\in SF^\ast(n)$ by $A \succ B$ it is meant that the error contributed by any element $a\in A$ in the Voronoi region of $a$ is larger than the error contributed by any element $b \in B$ in the Voronoi region of $b$. Similarly, we define the order relation $ \succ $ on the set $SF(n)$.

\begin{lemma} For a given positive integer $n\geq 4$, let $S(a(4)) \in SF^\ast(n)$. Then, $S(a(4)) \succ B$ for any $B \in SF^\ast(n)\setminus \set{S(a(4))}$.
\end{lemma}
\begin{proof} We first prove that $S(a(4)) \succ S(a(\ell))$ for $5\leq \ell\leq n$. In order to that, for $5\leq \ell\leq n$, we first prove the following inequality:
 \begin{equation} \label{eq61}
\frac{18^{a(\ell)-a(4)}}{75^{\ell-4}}>1.
 \end{equation}
  First, assume that $\ell=2k+1$ for some $k\geq 2$. Then,
$
\frac{18^{a(\ell)-a(4)}}{75^{\ell-4}}=\frac{18^{a(2k+1)-a(4)}}{75^{2k+1-4}}=\frac{18^{3k-4}}{75^{2k-3}}=\Big(\frac{18^3}{75^2}\Big)^{(k-1)} \frac {75}{18}>1.
$
Next, assume that $\ell=2k$ for some $k\geq 3$. Then, $
\frac{18^{a(\ell)-a(4)}}{75^{\ell-4}}=\frac{18^{a(2k)-a(4)}}{75^{2k-4}}=\frac{18^{3k-2-4}}{75^{2k-4}}=\Big(\frac{18^3}{75^2}\Big)^{(k-2)}>1.
$
Thus, the inequality is proved. The distortion error due to any element from the set $S(a(\ell))$ is given by $\frac 1{75^{n-\ell}}\frac 13 \frac 1{18^{a(\ell)}} W$. On the other hand, the distortion error due to any element from the set $S(a(4))$ is given by $\frac 1{75^{n-4}}\frac 13 \frac 1{18^{a(4)}} W$. Thus, $\frac 1{75^{n-4}}\frac 13 \frac 1{18^{a(4)}} W>\frac 1{75^{n-\ell}}\frac 13 \frac 1{18^{a(\ell)}} W$ is true if $\frac{18^{a(\ell)-a(4)}}{75^{\ell-4}}>1$, which is clearly true by \eqref{eq61}. Next, take $2\leq \ell \leq \lfloor \frac n 2\rfloor$. Then, the distortion error due to the set $S^{(2)}(a(2\ell))$ is given by $\frac 1{75^{n-2\ell}} \frac 13 \frac 1{18^{a(2\ell)+1}}W=\frac 1{75^{n-2\ell}}\frac 13 \frac 1{18^{3\ell-1}}W$. Thus, $\frac 1{75^{n-4}}\frac 13 \frac 1{18^{a(4)}} W>\frac 1{75^{n-2\ell}}\frac 13 \frac 1{18^{3\ell-1}}W$ is true if $\frac{18^{3\ell-5}}{75^{2\ell-4}}>1$, i.e., if $(\frac{18^3}{75^2})^{(\ell-2)} 18>1$, which is clearly true as $\ell\geq 2$. Similarly, we can show that $S(a(4)) \succ S^{(2)}(0)$.
Now, take $1\leq \ell\leq 3$. Then,  $S(a(4)) \succ S(a(\ell))$ is true since
$\frac 1{75^{n-4}} \frac 13 \frac 1{18^{a(4)}}W>\frac 1{75^{n-\ell}} \frac 13 \frac 1{18^{\ell}} W$. Similarly, $S(a(4)) \succ S(0)$. Thus, the assertion follows.
\end{proof}

\begin{lemma} \label{lemma19}  Let $0\leq k\leq n$. Then,
\begin{itemize}
\item[(i)] $S(a(2k))\succ S(a(2k+2))$ for all $k\geq 2$.
\item[(ii)] $S(a(82))\succ S(a(2k+1))\succ S(a(2k+82)) \succ S(a(2k+3))$ for all $k\geq 1$.
\item[(iii)] $S(a(83)) \succ S^{(2)}(a(2k)) \succ S(a(2k+160)) \succ S(a(2k+81))  \succ S^{(2)}(a(2k+2))$ for all $k\geq 2$.
\item[(iv)] $S(a(7))\succ  S(0) \succ S(a(88)) \succ S(a(9))$,  $S(a(81)) \succ S(2) \succ S(a(162)) \succ S(a(83))$, $S(a(121)) \succ S^{(2)}(a(42)) \succ S(a(202)) \succ S^{(2)}(0) \succ S(a(123)) \succ S^{(2)}(a(44))$, and \\ $\ S^{(2)} (a(80)) \succ S(a(240)) \succ S(1) \succ S(a(161)) \succ S^{(2)} (a(82))$.
\end{itemize}
\end{lemma}

\begin{proof}
$(i)$  For any $k\geq 2$, let $b\in S(a(2k))=\uu_{\go \in I^{n-2k}} S_\go(\ga_{2^{a(2k)}}(\gn))$. Then, $b=S_\go(T_\gt(\frac 12))$ for some $\go \in I^{n-2k}$ and $\gt \in I^{a(2k)}$. The error contributed by $b$ in its Voronoi region is given by
\[\int_{S_\go(T_\gt(L))}(x- S_\go(T_\gt(\frac 12)))^2 dP=\frac 1{75^{n-2k}} \frac 13 \frac 1{18^{a(2k)}}W.\]
Similarly, the error contributed by any element in the set $S(a(2k+2))$ is $\frac 1{75^{n-2k-2}} \frac 13 \frac 1{18^{a(2k+2)}}W$. Thus, $S(a(2k)) \succ S(a(2k+2))$ will be true if $\frac 1{75^{n-2k}} \frac 13 \frac 1{18^{a(2k)}}W>\frac 1{75^{n-2k-2}} \frac 13 \frac 1{18^{a(2k+2)}}W$, i.e., if
$1<\frac{18^{a(2k+2)-a(2k)} }{75^{2}}=\frac{18^3}{75^2}$, which is clearly true.

$(ii)$ $S(a(82)) \succ S(a(3))$ is true since $\frac 1{75^{n-82}} \frac 13 \frac 1{18^{a(82)}} W>\frac 1{75^{n-3}} \frac 13 \frac 1{18^{a(3)}} W$. For all $k\geq 1$, proceeding similarly as in $(i)$, we have $S(a(2k+1)) \succ S(a(2k+82))$ and $S(a(2k+82)) \succ S(a(2k+3))$. Thus, $(ii)$ is proved.

$(iii)$  $S(a(83)) \succ S^{(2)}(a(4))$ is true since $\frac 1{75^{n-83}} \frac 13 \frac 1{18^{a(83)}} W>\frac 1{75^{n-4}} \frac 13 \frac 1{18^{a(4)+1}} W$. Proceeding similarly as in $(i)$, we can show that $S^{(2)}(a(2k)) \succ S(a(2k+160))$, $S(a(2k+160)) \succ S(a(2k+81))$, and $S(a(2k+81)) \succ S^{(2)}(a(2k+2))$ for all $k\geq 2$. Thus $(iii)$ follows.

$(iv)$  Since $\frac 1{75^{n-7}} \frac 13 \frac 1{18^{a(7)}} W>\frac 1{75^n} V>\frac 1{75^{n-88}} \frac 13 \frac 1{18^{a(88)}} W>\frac 1{75^{n-9}} \frac 13 \frac 1{18^{a(9)}} W$, the inequality $S(a(7)) \succ S(0) \succ S(a(88))  \succ S(a(9))$ is true. Similarly we can show the other inequalities.
\end{proof}


\begin{remark} \label{remark5555}  The Lemma~\ref{lemma19} defines the order among the elements of the sets $SF^\ast(n)$ for all $n\geq 1 . $  Since it is helpful in obtaining the optimal sets of $F(n)$-means, we exhibit the following cases for later use:
\begin{itemize}
\item[(i)] $S(0)\succ S^{(2)}(0)\succ S(1)$ for $n=1$;
\item[(ii)] $S(0) \succ S(2) \succ S^{(2)}(0) \succ S(1)$ for $n=2$;
\item[(iii)] $S(a(3)) \succ S(0) \succ S(2) \succ S^{(2)}(0) \succ S(1)$ for $n=3$;
\item[(iv)] $S(a(4)) \succ S(a(3)) \succ S(0) \succ S(2) \succ S^{(2)}(a(4)) \succ S^{(2)}(0) \succ S(1)$ for $n=4 . $
\end{itemize}
\end{remark}

\begin{lemma} \label{lemma43}
Let $\ga_{F(n)}$, $SF(n)$, $SF^{(1)}(n)$, and $SF^{(2)}(n)$ be the sets as defined before. Then,
\[\ga_{F(n+1)}=\left(\uu_{A\in SF^{(1)}(n)}A^{(2)(2)}\right)\uu \left(\uu_{A\in SF^{(2)}(n)}A^{(2)}\right).\]
\end{lemma}

\begin{proof} Let $A \in SF^{(1)}(n)$. If $A=S(0)$, then $A^{(2)(2)}=S^{(2)(2)}(0)=\set{S_\go(\frac 12) :  \go \in I^{n+1}}\uu (\uu_{\go \in I^n}S_\go(\ga_{2}(\gn)))=\set{S_\go(\frac 12) :  \go \in I^{n+1}}\uu (\uu_{\go \in I^{n+1-1}}S_\go(\ga_{2}(\gn)))$. If $A=S(a(2\ell))$ for some $\ell \in \set{1, 2, \cdots, \lfloor \frac  n 2\rfloor}$, then
\[A^{(2)(2)}=\uu_{\go \in I^{n-2\ell}}S_\go(\ga_{2^{a(2\ell)+2}}(\gn))=\uu_{\go \in I^{n+1-(2\ell+1)}}S_\go(\ga_{2^{a(2\ell+1)}}(\gn)).\]
Next, let $A \in SF^{(2)}(n)$. Then, $A=S(1), S(2), \te{ or } S(a(2\ell+1))$ for some $\ell\in \D N$. If $A=S(1)$, then $A^{(2)}=\uu_{\go\in I^{n-1}} S_\go(\ga_{2^{1+1}}(\gn))=\uu_{\go\in I^{n+1-2}} S_\go(\ga_{2^{2}}(\gn))$, and similarly, we have  $S^{(2)}(2)=\uu_{\go\in I^{n+1-3}} S_\go(\ga_{2^{3}}(\gn))$. If $A=S(a(2\ell+1))$, then
\[A^{(2)}=\uu_{\go \in I^{n-(2\ell+1)}}S_\go(\ga_{2^{a(2\ell+1)+1}}(\gn))=\uu_{\go \in I^{n-(2\ell+1)}}S_\go(\ga_{2^{3\ell+1}}(\gn))=\uu_{\go \in I^{n-(2\ell+1)}}S_\go(\ga_{2^{a(2(\ell+1))}}(\gn)),\]
yielding, $A^{(2)}=\uu_{\go \in I^{n+1-(2\ell+1)}}S_\go(\ga_{2^{a(2(\ell+1))}}(\gn))$. Thus, we see that
\[\ga_{F(n+1)}=\left(\uu_{A\in SF^{(1)}(n)}A^{(2)(2)}\right)\uu \left(\uu_{A\in SF^{(2)}(n)}A^{(2)}\right),\]
which proves the assertion.
\end{proof}


\begin{lemma}\label{lemma71}  For $A, B\in SF (n)$ with $A \succ B$, the distortion error due to the set $ (SF(n)\setminus A)\uu A^{(2)}\uu B$ is less than the distortion error due to the set $ (SF(n)\setminus B)\uu B^{(2)}\uu A$.
\end{lemma}

\begin{proof} Let $V(\ga_{F(n)})$ be the distortion error due to the set $\ga_{F(n)}$ with respect to the condensation measure $P$. First take $A=S(a(k))$ and $B=S(a(k'))$ for some $3\leq k\neq k'\leq n$. Then, the distortion error due to the set $(\ga_{F(n)}\setminus A)\uu A^{(2)}\uu B$ is less than the distortion error due to the set $(\ga_{F(n)}\setminus B)\uu B^{(2)}\uu A$ is
\begin{align*}
& V(\ga_{F(n)})-\frac 1{75^{n-k}} \frac 13 \frac 1{9^{a(k)}}W+\frac 1{75^{n-k}} \frac 13 \frac 1{9^{a(k)+1}}W+\frac 1{75^{n-k'}} \frac 13 \frac 1{9^{a(k')}}W\\
& <V(\ga_{F(n)})-\frac 1{75^{n-k'}} \frac 13 \frac 1{9^{a(k')}}W+\frac 1{75^{n-k'}} \frac 13 \frac 1{9^{a(k')+1}}W+\frac 1{75^{n-k}} \frac 13 \frac 1{9^{a(k)}}W,
\end{align*}
yielding $\frac 1{75^{n-k}} \frac 13 \frac 1{9^{a(k)}}W>\frac 1{75^{n-k'}} \frac 13 \frac 1{9^{a(k')}}W$, which is clearly true since by the hypothesis $A \succ B$. Similarly, we can prove the lemma in each of the following cases:

 $(i)$ $A=S(a(k))$ and $B=S(k')$, where $3\leq k\leq n$ and $0\leq k'\leq 2$;

 $(ii)$ $A=S(k)$ and $B=S(a(k'))$, where $0\leq k\leq 2$ and $3\leq k'\leq n$;

  $(iii)$ $A=S(k)$ and $B=S(a(k'))$, where $0\leq k\leq 2$ and $0\leq k'\leq 2$.

  Thus, the proof of the lemma is complete.
\end{proof}

Using the similar technique as Lemma~\ref{lemma71}, the following lemma can be proved.

\begin{lemma}\label{lemma72}  For any two sets $A, B\in SF^\ast (n)$, let $A \succ B$. Then, the distortion error due to the set $ (SF^\ast(n)\setminus A)\uu A^{(2)}\uu B$ is less than the distortion error due to the set $ (SF^\ast(n)\setminus B)\uu B^{(2)}\uu A$.
\end{lemma}
\begin{remark}\label{remark45}
By Proposition~\ref{prop3}, we know that $\ga_{F(1)}$ is an optimal set of $F(1)$-means. Assume that $\ga_{F(n)}$ is an optimal set of $F(n)$-means for some $n\geq 1$. Let $A\in SF(n)$ be such that $A \succ B$ for any other $B\in SF(n)$. By Lemma~\ref{lemma71}, we deduce that if $A=S(0)=\set{S_\go(\frac 12) : \go \in I^n}$, then the set $(\ga_{F(n)}\setminus A)\uu\set{S_\go (\frac {19}{90}), S_\go(\frac {71}{90}) : \go \in I^n}$ is an optimal set of $F(n)-2^n+2^{n+1}$-means. If $A=\uu_{\go \in I^{n-\ell}} S_\go(\ga_{2^{\ell}}(\gn))$ for $1\leq \ell\leq 3$, then the set  $(\ga_{F(n)}\setminus A)\uu (\uu_{\go \in I^{n-\ell}} S_\go(\ga_{2^{\ell+1}}(\gn)))$ is an optimal set of $F(n)-2^n+2^{n+1}$-means. If $A=\uu_{\go \in I^{n-\ell}}S_\go(\ga_{2^{a(\ell)}}(\gn))$, then the set $(\ga_{F(n)}\setminus A)\uu (\uu_{\go \in I^{n-\ell}} S_\go(\ga_{2^{a(\ell)+1}}(\gn)))$ is an optimal set of $F(n)-2^{n-\ell} 2^{a(\ell)}+2^{n-\ell}2^{a(\ell)+1}$-means.
\end{remark}

\medskip

\begin{prop}\label{prop351}
Let $a(n)$ and $F(n)$ be the two sequences as defined by Definition~\ref{defi3}. Then, for any $n\geq 1$, the set $\ga_{F(n)}(P)$
is an optimal set of $F(n)$-means with quantization error given by
\[V_{F(n)}(P)=\left\{\begin{array} {ll}
\frac 13 \frac 1{9}W+\frac 2{75}V \ \  \te{ if } n=1,\\
 \frac 13 \frac 1{9^2}W+\frac 2{75}\frac 13 \frac 1{9}W+\left(\frac 2{75}\right)^{2}V  \ \  \te{ if } n=2,\\
 \sum_{\ell=3}^n\left(\frac 2{75}\right)^{n-\ell}\frac 13 \frac 1{9^{a(\ell)}}W+\left(\frac 2{75}\right)^{n-2}\frac 13 \frac 1{9^{2}}W+\left(\frac 2{75}\right)^{n-1}\frac 13 \frac 1{9}W+\left(\frac 2{75}\right)^{n}V  \ \ \te{ if } n\geq 3.\end{array} \right.\]
\end{prop}

\begin{proof} By Proposition~\ref{prop3}, for $n=1$ the set $\ga_{F(1)}$ is an optimal set of $F(1)$-means with quantization error $\frac 13 \frac 1{9}W+\frac 2{75}V$. We now show that $\ga_{F(n)}$ is an optimal set of $F(n)$-means for any $n\geq 2$. Consider the following cases:

Case 1. $n=2$.

Let $\ga$ be an optimal set of $F(2)$-means. Recall that $\ga$ does not contain any point from the open intervals $(\frac 15, \frac 25)$ and $(\frac 35, \frac 45)$. The distortion error due to the set $\ga_{F(2)}(P)=\ga_{2^2}(\gn) \uu \left(\uu_{\go \in I} S_\go(\ga_2(\gn))\right) \uu \set{S_\go(\frac 12) : \go \in I^2}$ is given by
\begin{equation*} \label{eq555} \int \min_{a\in \ga_{F(2)}(P)} (x-a)^2 dP= \frac 13 \frac 1{9^2}W+\frac 2{75}\frac 13 \frac 1{9}W+\left(\frac 2{75}\right)^{2}V=\frac{9283}{88695000}.
\end{equation*}
Since $V_{F(2)}(P)$ is the quantization error for $F(2)$-means, we have $V_{F(2)}(P)\leq \frac{9283}{88695000}.$ Since $\ga$ is an optimal set of $F(2)$-means, $F(2)$ is even, and $P$ is symmetric about the point $\frac 12$, $\ga$ must contain equal number of points from each of the sets $J_1$ and $J_2$ implying that $\ga$ contains even number of points from $L$. We now show that $\ga$ contains $2^2$ elements from $L$. Suppose that $\ga$ contains less than $2^2$ elements from $L$. Then, $\te{card}(\ga\ii L)=0$ or $2$. But, $\ga\ii L\neq \es$, and so $\te{card}(\ga\ii L)=2$. Hence,
\[\int \min_{a\in \ga_{F(2)}(P)} (x-a)^2 dP\geq \int_{L} \min_{a\in\ga\ii L}(x-a)^2 dP=\frac 13 \frac 19 W=\frac{1}{5400}>V_{F(2)}(P), \]
which is a contradiction. Next suppose that $\ga$ contains more than $2^2$ elements from $L$. As $\ga\ii J_i\neq \es$ for $i=1, 2$, we must have $\te{card}(\ga\ii L)=6, 8$, or $10$. Suppose that $\te{card}(\ga\ii L)=6$. Then, for $1\leq i\leq 2$, $\ga\ii J_i$ is an optimal set of three-means with respect to the image measure $P\circ S_i^{-1}$, which by Lemma~\ref{lemma4} yields that $S_i^{-1}(\ga\ii J_i)$ is an optimal set of three-means for $P$, and so by Proposition~\ref{prop2}, we have
$S_i^{-1}(\ga\ii J_i)=\set{S_1(\frac 12), \frac 12, S_2(\frac 12)}$, i.e., $\ga\ii J_1=\set{S_{11}(\frac 12), S_1(\frac 12), S_{12}(\frac 12)}$, and $\ga\ii J_2=\set{S_{21}(\frac 12), S_2(\frac 12), S_{22}(\frac 12)}$. Then, the distortion error is
$V_{F(2)}(P)\geq2\Big(\int_{J_{11}}(x-S_{11}(\frac 12))^2 dP+\int_{S_1(L)}(x-S_1(\frac 12))^2dP+\int_{J_{12}}(x-S_{12}(\frac 12))^2 dP\Big)=2\Big(\frac 1{75^2} V+\frac 1{75} \frac 13 W+\frac 1{75^2} V\Big)=\frac{203}{1642500}>V_{F(2)}(P),$
which is a contradiction. Similarly, we can show that if $\te{card}(\ga\ii L)=8$ or $10$, contradiction will arise. Hence, $\te{card}(\ga\ii L)=2^2$, which implies that $\ga$ contains $F(1)$ elements from each of the sets $J_1$ and $J_2$. Thus, we see that in this case $\ga=\ga_{F(2)}(P)$ with quantization error $V_{F(2)}(P)=\frac 13 \frac 1{9^2}W+\frac 2{75}\frac 13 \frac 1{9}W+\left(\frac 2{75}\right)^{2}V$.

Case~3. $n=3$.

By Case 1, we know that $\ga_{F(2)}=S(0)\uu S(1)\uu S(2)$ is an optimal set of $F(2)$-means,  where $S(0), S(1), S(2)\in SF(2)$.  By Remark~\ref{remark5555},  $S(0) \succ S(2) \succ S^{(2)}(0) \succ S(1)$. Thus, by Remark~\ref{remark45}, the set $(\ga_{F(2)}\setminus S(0))\uu S^{(2)}(0)=S(1)\uu S(2)\uu  S^{(2)}(0)$ is an optimal set of $2^2+2^2+2^3$-means. Similarly, the set $S(1)\uu S^{(2)}(2)\uu  S^{(2)}(0)$ is an optimal set of $2^2+2^3+2^3$-means, the set $S(1)\uu S^{(2)}(2)\uu  S^{(2)(2)}(0)$ is an optimal set of $2^2+2^3+2^3+2^3$-means, the set $S^{(2)}(1)\uu S^{(2)}(2)\uu  S^{(2)(2)}(0)$ is an optimal set of $2^3+2^3+2^3+2^3=F(3)$-means. By Lemma~\ref{lemma43}, it is known that $\ga_{F(3)}=S^{(2)}(1)\uu S^{(2)}(2)\uu  S^{(2)(2)}(0)$. Thus, $\ga_{F(3)}$ is an optimal set of $F(3)$-means with quantization error same as it is given in the hypothesis for $n=3$.

Case~4. $n\geq 4$.

Let $\ga_{F(n)}$ be an optimal set of $F(n)$-means for some $n\geq 3$. We need to show that $\ga_{F(n+1)}$ is an optimal set of $F(n+1)$-means. We have
$\ga_{F(n)}=\uu_{A\in SF(n)}A$. In the first step, let $A(1) \in SF(n)$ be such that $A(1) \succ B$ for any other $B\in SF(n)$. Then, by Remark~\ref{remark45}, the set  $(\ga_{F(n)}\setminus A(1)) \uu A^{(2)}(1)$ gives an optimal set of $F(n)-\te{card}(A(1))+\te{card}(A^{(2)}(1))$-means. In the 2nd step, let $A(2) \in (SF(n)\setminus \set{A(1)})\uu \set{A^{(2)}(1)}$ be such that $A(2) \succ B$ for any other set $B\in  (SF(n)\setminus \set{A(1)})\uu \set{A^{(2)}(1)}$. Then, using the similar technique as Lemma~\ref{lemma71}, we can show that the distortion error due to the following set:
\begin{equation} \label{eq44} \Big(((\ga_{F(n)}\setminus A(1)) \uu A^{(2)}(1))\setminus A(2)\Big)\uu A^{(2)}(2)
\end{equation}
 with cardinality $F(n)-\te{card}(A(1))+\te{card}(A^{(2)}(1))-\te{card}(A(2))+\te{card}(A^{(2)}(2))$ is smaller than the distortion error due to the set obtained by replacing $A(2)$ in the set \eqref{eq44} by any other set $A'(2)$ having the same cardinality as $A(2)$. In other words, $\Big(((\ga_{F(n)}\setminus A(1)) \uu A^{(2)}(1))\setminus A(2)\Big)\uu A^{(2)}(2)$ forms an optimal set of
$F(n)-\te{card}(A(1))+\te{card}(A^{(2)}(1))-\te{card}(A(2))+\te{card}(A^{(2)}(2))$-means. Proceeding inductively in this way, up to $(n+1+2\lfloor \frac n 2\rfloor)$ steps, we can see that $\ga_{F(n+1)}=\left(\uu_{A\in SF^{(1)}(n)}A^{(2)(2)}\right)\uu \left(\uu_{A\in SF^{(2)}(n)}A^{(2)}\right)$ forms an optimal set of $F(n+1)$-means.

The $F(n)$-th quantization error for all $n\geq 3$ is given by
\begin{align*}
&\int \min_{a\in \ga_{F(n)}(P)} (x-a)^2 dP\\
&=\frac 13 V_{2^{a(n)}}(\gn)+\sum_{\go \in I}\sum_{\gt \in I^{a(n-1)}} \int _{S_\go T_\gt(L)}(x-S_\go T_\gt(\frac 12))^2 dP+\sum_{\go \in I^2}\sum_{\gt \in I^{a(n-2)}} \int _{S_\go T_\gt(L)}(x-S_\go T_\gt(\frac 12))^2 dP\\
&+\cdots +\sum_{\go \in I^{n-5}}\sum_{\gt \in I^{a(5)}} \int _{S_\go T_\gt(L)}(x-S_\go T_\gt(\frac 12))^2 dP+\sum_{\go \in I^{n-4}}\sum_{\gt \in I^{a(4)}} \int _{S_\go T_\gt(L)}(x-S_\go T_\gt(\frac 12))^2 dP\\
&+\sum_{\go \in I^{n-3}}\sum_{\gt \in I^{a(3)}} \int _{S_\go T_\gt(L)}(x-S_\go T_\gt(\frac 12))^2 dP+\sum_{\go \in I^{n-2}}\sum_{\gt \in I^2} \int _{S_\go T_\gt(L)}(x-S_\go T_\gt(\frac 12))^2 dP\\
&+\sum_{\go \in I^{n-1}}\sum_{\gt \in I} \int _{S_\go T_\gt(L)}(x-S_\go T_\gt(\frac 12))^2 dP +\sum_{\go \in I^{n}} \int_{J_\go}(x-S_\go (\frac 12))^2 dP\\
&=\frac 13 \frac 1{9^{a(n)}} W+\frac 2{75}\frac 13 \frac 1{9^{a(n-1)}}W+\Big(\frac 2{75}\Big)^2\frac 13 \frac 1{9^{a(n-2)}}W+\cdots +\Big(\frac 2{75}\Big)^{n-5}\frac 13 \frac 1{9^{a(5)}}W\\
&+\Big(\frac 2{75}\Big)^{n-4}\frac 13 \frac 1{9^{a(4)}}W+\Big(\frac 2{75}\Big)^{n-3}\frac 13 \frac 1{9^{a(3)}}W+\Big(\frac 2{75}\Big)^{n-2}\frac 13 \frac 1{9^{2}}W+\Big(\frac 2{75}\Big)^{n-1}\frac 13 \frac 1{9}W+\Big(\frac 2{75}\Big)^{n}V\\
&=\sum_{\ell=3}^n\Big(\frac 2{75}\Big)^{n-\ell}\frac 13 \frac 1{9^{a(\ell)}}W+\Big(\frac 2{75}\Big)^{n-2}\frac 13 \frac 1{9^{2}}W+\Big(\frac 2{75}\Big)^{n-1}\frac 13 \frac 1{9}W+\Big(\frac 2{75}\Big)^{n}V.
\end{align*}
Thus, the proof of the proposition is complete.
\end{proof}

Using Proposition~\ref{prop351} and the recursive properties of canonical sequences, by routine calculations the following lemma is obtained.

\begin{lemma}\label{lemma111} For all $k\geq 2$, we have
\begin{align*} V_{F(2k+1)}& =\frac 1{9^{a(2k+1)}} \frac{79}{7224} \Big(1-\Big(\frac {324}{625}\Big)^k\Big) - \frac{3571}{44347500}\Big(\frac 2{75}\Big)^{2k-1}, \te{ and } \\
 V_{F(2(k+1))}&= \frac 1 3 \frac 1{9^{a(2(k+1))}} W+\frac 2{75} V_{F(2k+1)}.
\end{align*}
\end{lemma}

\begin{remark}
By Proposition~\ref{prop351}, we see that if $\ga_{F(n)}$ is an optimal set of $F(n)$-means for some $n\geq 4$, then $\ga_{F(n)}$ contains $F(n-1)$ elements from each of $J_1$ and $J_2$, and $2^{a(n)}$ elements from $L$. Thus, by Proposition~\ref{prop55}, we have
$V_{F(n)}(P)=  \frac 1 3 \frac 1{9^{a(n)}} W+\frac 2 {75}  V_{F(n-1)}(P)$.
\end{remark}

\subsection{Calculation of Optimal sets of $m$-means for all $m\in \D N$} \label{subsection}

Now, we give the description of how to calculate the optimal sets of $m$-means for any $m \in \D N$.

For $1\leq m\leq 4$, the optimal sets of $m$-means and the $m$th quantization error are known by Lemma~\ref{lemma3}, Proposition~\ref{prop1}, Proposition~\ref{prop2}, and Proposition~\ref{prop3}. If $m=F(n)$ for some positive integer $n$, then they are known by Proposition~\ref{prop351}. Let $m\in \D N$ be such that $F(n_m)<m<F(n_m+1)$ for some positive integer $n_m$. Notice that $SF^\ast(n_m)$ contains $(n_m+1+\lfloor \frac {n_m}2 \rfloor)$ elements. After rearranging the elements of $SF^\ast(n_m)$ write
\[SF^\ast(n_m)=\set{C(0), C(1), C(2), \cdots, C(n_m+1+\lfloor \frac {n_m}2 \rfloor)},\]
where $C(0) \succ C(1) \succ C(2) \succ \cdots \succ C(n_m+1+\lfloor \frac {n_m}2 \rfloor)$. Let $\ell_m\in \D N$ be such that
\[F(n_m)+\te{card}(C(0))+\cdots +\te{card}(C(\ell_m))\leq m<F(n_m)+\te{card}(C(0))+\cdots +\te{card}(C(\ell_m+1)).\]
If $m=F(n_m)+\te{card}(C(0))+\cdots +\te{card}(C(\ell_m))$, then the set
\begin{align*} \ga_m(P)&=\Big(\ga_{F(m)}\setminus (C(0)\uu C(1)\uu\cdots\uu C(\ell_m))\Big)\\
& \uu \Big((C^{(2)}(0)\uu C^{(2)}(1)\uu\cdots\uu C^{(2)}(\ell_m))\setminus (C(0)\uu C(1)\uu\cdots\uu C(\ell_m))\Big)
\end{align*}
is a unique optimal set of $m$-means. Let $F(n_m)+\te{card}(C(0))+\cdots +\te{card}(C(\ell_m))< m<F(n_m)+\te{card}(C(0))+\cdots +\te{card}(C(\ell_m+1))$. Let $\gb_m\sci C(\ell_m+1)$ with $\te{card}(\gb_m)=m-(F(n_m)+\te{card}(C(0))+\cdots +\te{card}(C(\ell_m)))$. The following three cases can arise:

Case~1. $C(\ell_m+1)=\uu_{\go\in I^{n-k}}S_\go(\ga_{2^{a(k)}}(\gn))$.

Write $
\gb_{m, 1}:=\set{\go \in I^{n-k} : S_\go T_\gt(\frac 12)\in \gb_m \te{ for some } \gt\in I^{a(k)}},$ and
$\gb_{m, 2}:=\set{\gt \in I^{a(k)} : S_\go T_\gt(\frac 12)\in \gb_m \te{ for some } \go\in I^{n-k}}$. Set
\[\gb_m^\ast:=\set{S_\go T_\gt T_1(\frac 12) : \go\in\gb_{m, 1} \te{ and } \gt\in \gb_{m, 2}} \uu \set{S_\go T_\gt T_2(\frac 12) : \go\in\gb_{m, 1} \te{ and } \gt\in \gb_{m, 2}}.\]

Case~2. $C(\ell_m+1)=\set{S_\go(\frac 12) : \go \in I^{n_m}}$.

Write $\gb_{m, 1}:=\set{\go \in I^{n_m} : S_\go(\frac 12) \in \gb}$, and then $\gb_m^\ast:=\mathop{\uu}\limits_{\go\in \gb_{m,1}} \set{S_\go(\frac {19}{90}), S_\go(\frac{71}{90})}$.

Case~3. $C(\ell_m+1)=\mathop{\uu}\limits_{\go\in I^{n_m}} \set{S_\go(\frac {19}{90}), S_\go(\frac{71}{90})}$.

Write $\gb_{m, 1}:=\set{\go \in I^{n_m} : S_\go(\frac {19}{90}) \in \gb_m \te{ or } S_\go(\frac{71}{90}) \in \gb_m}$, and write
\[\gb_m^\ast:=\mathop{\uu}\limits_{\set{\go \in \gb_{m, 1} \&  S_\go(\frac {19}{90})\in \gb_m}} \set{S_{\go1} (\frac 12), S_{\go}T_1(\frac 12) }\uu \mathop{\uu}\limits_{\set{\go \in \gb_{m, 1} \&  S_\go(\frac {71}{90})\in \gb_m}} \set{ S_{\go}T_2(\frac 12), S_{\go2} (\frac 12) }.\]

Let $\gb_m^\ast$ be the set that arises either in Case~1, Case~2, or in Case~3. Then, the set
\begin{align*} \ga_m(P)& =\Big(\Big(\ga_{F(m)}\setminus (C(0)\uu C(1)\uu\cdots\uu C(\ell_m))\Big)\\
&\uu \Big((C^{(2)}(0)\uu C^{(2)}(1)\uu\cdots\uu C^{(2)}(\ell_m))\setminus (C(0)\uu C(1)\uu\cdots\uu C(\ell_m))\Big)\setminus \gb_m \Big) \uu \gb_m^\ast
\end{align*}
is an optimal set of $m$-means. The number of such sets in any of the above cases is given by $^{\te{card}(C(\ell_m+1))} C_{\te{card}(\gb_m)}$,
where $^uC_v =\begin{pmatrix} u \\ v \end{pmatrix} $ represent the binomial coefficient.

The following two examples illustrate the computations described above.

\begin{example} Let $m=21$. Since $F(2)<m<F(3)$, we have $n_m=2$. Since $S(0) \succ S(2) \succ S^{(2)}(0) \succ S(1)$, we have
$SF^\ast(n_m)=SF^\ast(2)=\set{C(0), C(1), C(2), C(3)},$
where $C(0)=S(0)$, $C(1)=S(2)$, $C(2)=S^{(2)}(0)$, and $C(3)=S(1)$. Again, $F(2)+\te{card}(C(0))+\te{card}(C(1))<m<F(2)+\te{card}(C(0))+\te{card}(C(1))+\te{card}(C(2))$ implying $\ell_m=1$, and so
$C(\ell_m+1)=S^{(2)}(0)$. Take $\gb_m=\set{S_{11}(\frac{19}{90}), S_{12}(\frac{19}{90}), S_{21}(\frac {71}{90})}$, where $\gb_m\sci S^{(2)}(0)$ with $\te{card}(\gb_m)=m-(F(2)+\te{card}(C(0))+\te{card}(C(1)))=23-20=3$. Then,
\[\gb_m^\ast=\set{S_{111}(\frac 12), S_{11}T_1(\frac 12), S_{121}(\frac 12), S_{12}T_1(\frac 12), S_{21}T_2(\frac 12), S_{212}(\frac 12)}.\]
Hence,
\begin{align*}
\ga_{23}(P)=&\Big(\Big((\ga_{F(2)}\setminus (C(0)\uu C(1)\Big)\uu \Big((C^{(2)}(0)\uu C^{(2)}(1)) \setminus(C(0)\uu C(1))\Big)\setminus \gb_m\Big)\uu \gb_m^\ast\\
&=\Big(S(1)\uu S^{(2)}(0)\uu S^{(2)}(2)\setminus \gb_m\Big)\uu \gb_m^\ast\\
&=\mathop{\uu}\limits_{\go\in I} S_\go(\ga_2(\gn)) \uu \ga_{2^3}(\gn) \uu \set{S_{11}(\frac {71}{90}), S_{12}(\frac {71}{90}), S_{21}(\frac {19}{90}), S_{22}(\frac {19}{90}), S_{22}(\frac {71}{90})}\\
&\qquad \qquad \uu \set{S_{111}(\frac 12), S_{11}T_1(\frac 12), S_{121}(\frac 12), S_{12}T_1(\frac 12), S_{21}T_2(\frac 12), S_{212}(\frac 12)},
\end{align*}
The number of optimal sets of 23-means is given by $^8 C_3 =56.$
\end{example}

\begin{example} Let $m=31$. As in the previous example, we have $n_m=2$, $C(0)=S(0)$, $C(1)=S(2)$, $C(2)=S^{(2)}(0)$, and $C(3)=S(1)$. Since $F(2)+\te{card}(C(0))+\te{card}(C(1))+\te{card}(C(2))=12+4+4+8=28<m<32=F(2)+\te{card}(C(0))+\te{card}(C(1))+\te{card}(C(2))+\te{card}(C(3))$, we have $\ell_m=2$, and so
$C(\ell_m+1)=S(1)$. Take $\gb_m=\set{S_1T_1(\frac 12), S_1T_2(\frac 12), S_2T_1(\frac 12)}$, where $\gb_m\sci S(1)$ with $\te{card}(\gb_m)=m-28=3$. Then,
\[\gb_m^\ast=\set{S_1T_1T_1(\frac 12), S_1T_1T_2(\frac 12), S_1T_2T_1(\frac 12), S_1T_2T_2(\frac 12), S_2T_1T_1(\frac 12), S_2T_1T_2(\frac 12)}.\]
Hence,
\begin{align*}
\ga_{23}(P)=&\Big(\Big((\ga_{F(2)}\setminus (C(0)\uu C(1)\uu C(2)\Big)\\
&\uu \Big((C^{(2)}(0)\uu C^{(2)}(1)\uu C^{(2)}(2)) \setminus(C(0)\uu C(1)\uu C(2))\Big)\setminus \gb_m\Big)\uu \gb_m^\ast\\
&=\Big((S(1)\uu S^{(2)}(2)\uu S^{(2)(2)}(0))\setminus \gb_m\Big)\uu \gb_m^\ast\\
&=\ga_{2^3}(\gn)\uu (\mathop{\uu}\limits_{\go\in I^2} S_\go(\ga_{2}(\gn)))\uu \set{S_\go(\frac 12) : \go \in I^3} \uu \set{S_2T_2(\frac 12)}\uu \gb_m^\ast
\end{align*}
is an optimal set of $31$-means. The number of optimal sets of 31-means is given by $^4 C_3 =4.$
\end{example}

\subsection{Asymptotics for the $n$th quantization error $V_n(P)$}
Having the optimal sets and the corresponding quantization errors are explicitly known, we now turn to the investigation of the quantization dimension and the quantization coefficients for the condensation measure $P$. If $D(\gn)$ is the quantization dimension of $\gn$, then $(\frac 12 (\frac 13)^2)^{\frac {D(\gn)}{2+D(\gn)}}+(\frac 12 (\frac 13)^2)^{\frac {D(\gn)}{2+D(\gn)}}=1$ (see \cite{GL2}), which yields $D(\gn)=\frac{\log 2}{\log 3}=\gb$, where $\gb$ is the Hausdorff dimension of the Cantor set generated by the similarity maps $T_1$ and $T_2$ (i.e., $\gb$ satisfies $(\frac 1 3)^\gb+(\frac 13)^\gb=1$).

\begin{theorem} \label{Th2}
Let $P$ be the condensation measure associated with the self-similar measure $\gn$. Then, $\lim_{n\to \infty} \frac{2 \log n}{\log V_n(P)}=D(\nu) $, i.e., the quantization dimension $D(P)$ of the measure $P$ exists and is equal to $D(\nu)$.
\end{theorem}

\begin{proof}
For $n\in \D N$, $n\geq 3$,  let $\ell(n)$ be the least positive integer such that $F(2\ell(n)+1)\leq n<F(2(\ell(n)+1)+1)$. Then,
$V_{F(2(\ell(n)+1)+1)}<V_n\leq V_{F(2\ell(n)+1)}$. Thus, we have
\begin{align*}
\frac {2\log\left(F(2\ell(n)+1)\right)}{-\log\left(V_{F(2(\ell(n)+1)+1)}\right)}< \frac {2\log n}{-\log V_n}< \frac {2\log\left(F(2(\ell(n)+1)+1)\right)}{-\log\left(V_{F(2\ell(n)+1)}\right)}
\end{align*}
By Lemma~\ref{lemma56} and Lemma~\ref{lemma111}, we have  $F(2\ell(n)+1)=2\cdot 2^{\ell(n)}+3\cdot 2^{3\ell(n)}$, and
$V_{F(2(\ell(n)+1)+1)}=9^{-3(\ell(n)+1)} \frac{79}{7224} \Big(1-\Big(\frac {324}{625}\Big)^{\ell(n)+1}\Big) - \frac{3571}{44347500}\Big(\frac 2{75}\Big)^{2\ell(n)+1}$. Notice that  $\ell(n)\to \infty$ whenever $n\to\infty$.

Again,
\begin{align*}
&\lim_{\ell(n)\to \infty} \frac {\log\left(F(2\ell(n)+1)\right)}{-\log\left(V_{F(2(\ell(n)+1)+1)}\right)} \Big ( \frac{\infty}{\infty} \te{ form} \Big)=\lim_{\ell(n)\to \infty} \frac {2\cdot 2^{\ell(n)}\log 2+9\cdot 2^{3\ell(n)}\log 2}{2\cdot 2^{\ell(n)}+3\cdot 2^{3\ell(n)}}\\
& \cdot\frac{9^{-3(\ell(n)+1)} \frac{79}{7224} \Big(1-\Big(\frac {324}{625}\Big)^{\ell(n)+1}\Big) - \frac{3571}{44347500}\Big(\frac 2{75}\Big)^{2\ell(n)+1}}{9^{-3(\ell(n)+1)}\Big(3\log 9 \frac{79}{7224} \Big(1-\Big(\frac {324}{625}\Big)^{\ell(n)+1}\Big)+\frac{79}{7224} \Big(\frac {324}{625}\Big)^{\ell(n)+1} \log\frac {324}{625}\Big) + \frac{3571}{44347500}\Big(\frac 2{75}\Big)^{2\ell(n)+1} 2\log \frac 2{75} }\\
&=\frac{\log 2}{2\log 3},
\end{align*}
and so
\[\lim_{\ell(n)\to \infty} \frac {2\log\left(F(2\ell(n)+1)\right)}{-\log\left(V_{F(2(\ell(n)+1)+1)}\right)} =\frac{\log 2}{\log 3}.  \te{ Similarly, } \lim_{\ell(n)\to \infty} \frac {2\log\left(F(2(\ell(n)+1)+1)\right)}{-\log\left(V_{F(2\ell(n)+1)}\right)}=\frac{\log 2}{\log 3}. \]
Thus, $\frac{\log 2}{\log 3}\leq \liminf_n \frac{2\log n}{-\log V_n}\leq \limsup_n \frac{2\log n}{-\log V_n}\leq\frac{\log 2}{\log 3}$ implying the fact that the quantization dimension of the measure $P$ exists and is equal to $D(\nu)=\gb$.
\end{proof}

\begin{theorem}\label{Th3}
$\gb$-dimensional quantization coefficient for the condensation measure $P$ does not exist; however, the $\gb$-dimensional lower and upper quantization coefficients for $P$ are finite and positive.
\end{theorem}

\begin{proof}
Since $\gb=\frac {\log 2}{\log 3}$ for any $k\geq 2$, we have
\begin{equation} \label{eq49} F(2k+1)^{\frac 2 \gb}=(2 \cdot 2^{2k}+3\cdot 2^{3k})^{\frac 2 \gb}=2^{\frac {6k}{\gb}}(\frac 2{2^k}+3)^{\frac 2\gb}=9^{3k}(\frac 2{2^k}+3)^{\frac 2\gb}.
\end{equation}
Similarly, $F(2(k+1))^{\frac 2 \gb}=9^{3k}(\frac 4{2^k}+8)^{\frac 2\gb}$. Moreover, $9^{a(2k+1)}=9^{3k}$ and $9^{a(2(k+1))}=9^{3(k+1)-2}=9^{3k+1}$, and $9^{3k}(\frac 2{75})^{2k-1}=(\frac{324}{625})^{k}\frac{75}{2}$. Then, by Lemma~\ref{lemma111}, we have
\begin{align*}
F(2k+1)^{\frac 2 \gb}V_{F(2k+1)}(P)=(\frac 2{2^k}+3)^{\frac 2\gb}\Big(\frac{79}{7224} \Big(1-\Big(\frac {324}{625}\Big)^k\Big)- \frac{3571}{44347500}\Big(\frac{324}{625}\Big)^{k}\frac{75}{2}\Big)
\end{align*}
yielding
\begin{equation} \label{eq50} \mathop{\lim}\limits_{k\to \infty} F(2k+1)^{\frac 2 \gb}V_{F(2k+1)}(P)=3^{\frac 2\gb} \frac{79}{7224},
\end{equation}  and
\begin{align*}
F(2(k+1))^{\frac 2 \gb}V_{F(2(k+1))}(P)=(\frac 4{2^k}+8)^{\frac 2\gb}\Big(\frac 1{27} W+\frac 2{75} \frac{79}{7224} \Big(1-\Big(\frac {324}{625}\Big)^k\Big)- \frac{3571}{44347500}\Big(\frac{324}{625}\Big)^{k}\Big)
\end{align*}
yielding $\mathop{\lim}\limits_{k\to \infty} F(2(k+1))^{\frac 2 \gb}V_{F(2(k+1))}(P)=8^{\frac 2 \gb}\Big(\frac 1{27} W+\frac 2{75} \frac{79}{7224}\Big)$.
Since $(F(2k+1)^{\frac 2 \gb}V_{F(2k+1)}(P))_{k\geq 2}$ and $(F(2(k+1))^{\frac 2 \gb}V_{F(2(k+1))}(P))_{k\geq 2}$ are two subsequences of $(n^{\frac 2\gb}V_n(P))_{n\in \D N}$ having two different limits, we can say that the sequence $(n^{\frac 2\gb}V_n(P))_{n\in \D N}$ does not converge, in other words, the $\gb$-dimensional quantization coefficient for $P$ does not exist. For $n\in \D N$, $n\geq 3$,  let $\ell(n)$ be the least positive integer such that $F(2\ell(n)+1)\leq n<F(2(\ell(n)+1)+1)$. Then,
$V_{F(2(\ell(n)+1)+1)}<V_n\leq V_{F(2\ell(n)+1)}$ implying
\[F(2\ell(n)+1)^{\frac 2\gb}V_{F(2(\ell(n)+1)+1)}<n^{\frac 2\gb} V_n< F(2(\ell(n)+1)+1)^{\frac 2\gb} V_{F(2\ell(n)+1)}. \]
As $\ell(n)\to \infty$ whenever $n\to \infty$, by equation~\eqref{eq49}, we have
\[\lim_{n\to \infty} \frac{F(2\ell(n)+1)}{F(2(\ell(n)+1)+1)}=\lim_{\ell(n)\to \infty}\frac{9^{3\ell(n)}(\frac 2{2^{\ell(n)}}+3)^{\frac 2\gb}}{9^{3{\ell(n)+1}}(\frac 2{2^{\ell(n)+1}}+3)^{\frac 2\gb}}=\frac 1 9.\]
Hence, using \eqref{eq50}, we have
\[\lim_{n\to \infty} F(2\ell(n)+1)^{\frac 2\gb}V_{F(2(\ell(n)+1)+1)}=\frac 19 \lim_{\ell(n)\to \infty} F(2(\ell(n)+1)+1)^{\frac 2\gb}V_{F(2(\ell(n)+1)+1)}=\frac 19 \Big(3^{\frac 2\gb} \frac{79}{7224}\Big),\]
and similarly,
\[\lim_{n\to \infty} F(2(\ell(n)+1)+1)^{\frac 2 \gb} V_{F(2\ell(n)+1)}=9 \lim_{\ell(n)\to \infty} F(2\ell(n)+1)^{\frac 2\gb}V_{F(2\ell(n)+1)}= 9 \Big(3^{\frac 2\gb} \frac{79}{7224}\Big), \]
yielding the fact that $\frac 19 (3^{\frac 2\gb} \frac{79}{7224})\leq \mathop{\liminf}\limits_{n\to \infty} n^{\frac 2\gb} V_n\leq \mathop{\limsup}\limits_{n\to \infty} n^{\frac 2\gb} V_n\leq 9(3^{\frac 2\gb} \frac{79}{7224})$, i.e., $\gb$-dimensional lower and upper quantization coefficients for $P$ are finite and positive.
\end{proof}

Recall that the critical value of the condensation system under consideration is the number $\gk$ satisfying $(\frac 13 (\frac 15)^2)^{\frac {\gk}{2+\gk}}=\frac{1}{2}.$  Hence, $\gk=\frac{2\log 2}{\log75-\log 2} ; $ on the other hand, $D(\gn)=\frac{\log2}{\log 3}>\gk$. Therefore, by Theorem~\ref{Th2}, it follows that

\begin{prop} \label{prop35}
$D(P)=\max \set{\gk, D(\gn)} .$
\end{prop}

\section{Condensation measure $P$ with self-similar measure $\gn$ satisfying $D(\gn)<\gk$}\label{third}

 In this section, we study the optimal quantization for the condensation measure $P$ generated by the condensation system $(\set{S_1, S_2}, (\frac 13, \frac 13, \frac 13), \gn)$, where the self-similar measure $\gn$ is given by $\gn=\frac 12 \gn \circ T_1^{-1}+\frac 12\gn\circ T_2^{-1}$ with $T_1(x)=\frac{1}{7}x+\frac{12}{35}$, and $T_2(x)=\frac{1}{7}x+\frac{18}{35}$ for all $x\in \D R$ (i.e., the case $s=\frac{1}{7} $.)  From the general results obtained in Section~\ref{first}, we have

\begin{itemize}

\item $E(\nu)=\frac12,\ W=V(\nu)=\frac{3}{400};\ E(P)=\frac12, \ V=V(P)=\frac{131}{1168}, $
\item $\alpha_1=\{\frac12 \}, $  with $V_1=V(P), $
\item $\alpha_2 =\{\frac{43}{210}, \frac{167}{210} \} , \ V_2=\frac{321827}{12877200}, $ and
\item $ \alpha_3 =\{\frac{1}{10}, \frac{1}{2}, \frac{9}{10} \} , \ V_3=\frac{481}{87600}. $

\end{itemize}

Below, if the arguments of some statements are exactly the same as their counterparts in the previous section, we will omit their proofs to avoid repetition.

Notice that in this case, we have $V_{2^n}(\gn)=\frac 1{49^n} W$, and
\[ \int_{T_\go(L)} (x-x_0)^2 d\gn(x) =\frac 1{2^k} \Big(\frac 1{49^k} W  +(T_\go(\frac 12)-x_0)^2\Big).\]


\begin{prop1}  \label{prop300} The set $\set{S_1(\frac 12), T_1(\frac 12), T_2(\frac 12), S_2(\frac 12)}$ is an optimal set of four-means with quantization error $V_4=\frac{13057}{4292400}. $
\end{prop1}

\begin{prop1}  \label{prop00300} Let $\ga_n$ be an optimal set of $n$-means for all $n\geq 3$. Then, $\ga_n\ii J_1\neq \es$, $\ga_n \ii J_2\neq \es$, and $\ga_n\ii L\neq \es$. Moreover, $\ga_n$ does not contain any point from the open intervals $(\frac 15, \frac 25)$ and $(\frac 35, \frac 45)$.
\end{prop1}
\medskip

\subsection{Canonical sequence and optimal quantization}

In this section, the two canonical sequences are $\set{a(n)}_{n\geq 1}$ and  $\set{F(n)}_{n\geq 1}$ given by the following definition.
\begin{defi} \label{defi300}  The canonical sequences for this condensation system $(\mathcal{S}, \bold{p}, \nu)$ is defined as
$$ \aligned &a(n)=\left\{\aligned & a(1)=1  \\
 & a(n)=n-1 \ \text{for all}\ n\geq 2, \ \text{and}
 \endaligned \right. \\
& F(n)=(n+3) 2^{n-1} \ \text{for all} \ n\geq 1. \endaligned $$
\end{defi}

From these definitions it follows that, for any $n \geq 1, $ $2^{a(n+1)}+2F(n)=2^{n}+(n+3)2^n=(n+4)2^n=F(n+1);$ hence, we have

\begin{lemma}
For the canonical sequences $a(n)$ and $F(n)$, $F(n+1)=2^{a(n+1)}+2 F(n)$.
\end{lemma}
For $1\leq \ell\leq n$, write  $S(\ell):=\uu_{\go \in I^{n-\ell}} S_\go(\ga_{2^{a(\ell)}}(\gn))$ and $S^{(2)}(\ell):=\uu_{\go \in I^{n-\ell}} S_\go(\ga_{2^{a(\ell)+1}}(\gn))$. Notice that if $\ell=n$, then $S(n)= \ga_{2^{a(n)}}(\gn)$. Moreover, write
\begin{align*} S(0):&=\set{S_\go(\frac 12) : \go\in I^n},  \\
 S^{(2)}(0):&=\set{S_\go(\ga_2(P)) : \go \in I^{n}}=\set{S_\go(\frac {43}{210}), S_\go(\frac{167}{210}) : \go \in I^{n}},
\te{ and } \\
S^{(2)(2)}(0):&=\uu_{\go \in I^{n}} S_\go(\ga_{2^{a(1)}}(\gn))\uu \set{S_\go(\frac 12) : \go \in I^{n+1}}.
\end{align*}
 For any $\ell\in \D N\uu \set{0}$, if $A:=S(\ell)$, we identify $S^{(2)}(\ell)$ and $S^{(2)(2)}(\ell)$ respectively by $A^{(2)}$ and $A^{(2)(2)}$.  For $n\in\D N$, set
\begin{equation} \label{eq67} \ga_{F(n)}:=S(n)\uu S(n-1)\uu S(n-2)\uu \cdots \uu S(1)\uu S(0), \ \text{and} \end{equation}
\[SF(n):=\set{S(n), S(n-1), S(n-2), \cdots, S(1), S(0)}.\]
In addition, write
\begin{align*} \label{eq35} SF^\ast(n): =\set{S(n), S(n-1), \cdots,  S(0),  S^{(2)}(0)}.
\end{align*}
The order relation $\succ $ on the elements of $SF^\ast(n)$ and $SF(n)$ are defined analogously as it is defined on the elements of $SF^\ast(n)$ and $SF(n)$ in Section~\ref{first}.

\begin{remark}\label{remark100} By Definition~\ref{defi300},
$ \ga_{F(n)}=S_1(\ga_{F(n-1)})\uu \ga_{2^{a(n)}}(\gn)\uu S_2(\ga_{F(n-1)}). $

\end{remark}
\begin{lemma} \label{lemma1900}  Let $\succ $ be the order relation on $SF^\ast(n)$. Then,
\begin{align*}
S(2)\succ S(0)&\succ S(3)\succ S(4)\succ \cdots\succ S(11) \succ S^{(2)}(0)\succ S(12)\\
&\succ S(13)\succ \cdots\succ S(18)\succ S(1)\succ S(19)\succ S(20)\succ \cdots.
\end{align*}
\end{lemma}
\begin{proof} For any $n\geq k\geq 1$, the distortion error due to any element in the set $S(k)$ is given by $\frac 1{75^{n-k}}\frac 13 \frac 1{2^{a(k)}} \frac 1{49^{a(k)}}W$. On the other hand, the distortion error due to to the set $S(0)$ and $S^{(2)}(0)$ are, respectively, given by $\frac{1}{75^n}V$ and $\frac 1{75^n} \frac 1 2 V_2$. Hence, $S(2)\succ S(0)\succ S(3)$ will be true if $\frac 13 \frac {75^2} {98} W>V>\frac 13 \frac {75^3} {98^2} W$ which is clearly true. Thus, $S(2)\succ S(0)\succ S(3)$. For $n>k\geq 2$, the inequality $S(k)\succ S(k+1)$ is true if $1>\frac {75}{98}$ which is obvious. Moreover, since $\frac 13 \frac {75^{11}} {98^{10}} W>\frac 12V_2>\frac 13 \frac {75^{12}} {98^{11}} W$, we have $S(11)\succ S^{(2)}(0)\succ S(12)$. Again, $\frac{75^{18}}{98^{17}}>\frac{75}{98}>\frac{75^{19}}{98^{18}}$ yields $S(18)\succ S(1)\succ S(19)$. Combining all these inequalities, we see that the lemma follows.
\end{proof}

\begin{remark} Lemma~\ref{lemma1900} implies that if $n=1$, then $S(0)\succ S^{(2)}(0)\succ S(1)$; if $n=2$, then $S(2)\succ S(0)\succ S^{(2)}(0)\succ S(1)$; if $n=3$, then $S(2)\succ S(0)\succ S(3)\succ S^{(2)}(0)\succ S(1)$; and so on.
\end{remark}

From the definitions of $S(k)$ for all $0\leq k\leq n$ and $S^{(2)}(0)$ it follows that

\begin{lemma} \label{lemma4300}
Let $\ga_{F(n)}$ and $SF(n)$ be the sets as defined before. Then,
\[\ga_{F(n+1)}=\mathop{\uu}\limits_{k=1}^n S^{(2)}(k) \uu S^{(2)(2)}(0).\]
\end{lemma}

\begin{lemma}\label{lemma7100}  For any two sets $A, B\in SF^\ast(n)$, let $A\succ B$. Then, the distortion error due to the set $ (SF^\ast(n)\setminus A)\uu A^{(2)}\uu B$ is less than the distortion error due to the set $ (SF^\ast (n)\setminus B)\uu B^{(2)}\uu A$.
\end{lemma}

\begin{proof} We have $SF^\ast(n)=\set{S(n), S(n-1), \cdots, S(1), S(0), S^{(2)}(0)}$. Let $V_{SF^\ast(n)}$ be the distortion error due to the set $SF^\ast(n)$. First, take $A=S(k)$ and $B=S(k')$ for some $2\leq k< k'\leq n$. Then, by Lemma~\ref{lemma1900}, $A\succ B$. The distortion error due to the set $(SF^\ast(n)\setminus A)\uu A^{(2)}\uu B$ is given by
\begin{align}\label{eq4500}
&V_{SF^\ast(n)}-(\frac 2{75})^{n-k} \frac 13 \frac W{49^{a(k)}}+(\frac 2{75})^{n-k} \frac 13 \frac W{49^{a(k)+1}}+(\frac 2{75})^{n-k'} \frac 13 \frac W{49^{a(k')}}\notag\\
&=V_{SF^\ast(n)}-(\frac 2{75})^{n-k} \frac 13 \frac{48W}{49^{k}}+(\frac 2{75})^{n-k'} \frac 13 \frac W{49^{k'-1}}
\end{align}
Similarly, The distortion error due to the set $(SF^\ast(n)\setminus B)\uu B^{(2)}\uu A$ is
\begin{align} \label{eq4600}
V_{SF^\ast(n)}-(\frac 2{75})^{n-k'} \frac 13\frac{48W}{49^{k'}}+(\frac 2{75})^{n-k} \frac 13 \frac W{49^{k-1}}.
\end{align}
Thus, \eqref{eq4500} will be less than \eqref{eq4600} if $(\frac {98}{75})^{k'-k}>1$, which is clearly true since $k'>k$. Similarly, we can prove the lemma for any two elements $A, B\in SF^\ast(n)$. Thus, the proof of the lemma is complete.
\end{proof}

\begin{prop} \label{prop720}
 For any $n\geq 1$ the set $\ga_{F(n)}$ is an optimal set of $F(n)$-means for the condensation measure $P$ whose quantization error is given by
 \[V_{F(n)}:=V_{F(n)}(P)=\left\{\begin{array} {ll} \frac{13057}{4292400} & \te{ if } n=1, \\
  \frac{69071}{6170325}\left(\frac{2}{75}\right)^{n-1}-\frac{3}{368} \left(\frac{1}{49}\right)^{n-1} & \te{ if } n\geq 2.
 \end{array}\right.\]
\end{prop}

\begin{proof}
By Proposition~\ref{prop300}, the set $\ga_{F(1)}=\set{S_1(\frac 12), T_1(\frac 12), T_2(\frac 12), S_2(\frac 12)}$ is an optimal set of $F(1)$-means with quantization error $V_4=\frac{13057}{4292400}$.
Proceeding in the same way as Proposition~\ref{prop351}, we can show that for any $n\geq 2$, the set $\ga_{F(n)}$ forms an optimal set of $F(n)$-means, and the quantization error $V_{F(n)}$ is given by
\begin{align*}
&\int\min_{a\in \ga_{F(n)}}(x-a)^2 dP=\sum_{k=0}^{n-1}\sum_{\go\in I^{k}} \int_{L_\go}\min_{a\in S_\go(\ga_{2^{a(n-k)}}(\gn))}(x-a)^2 dP+\sum_{\go\in I^{n}}\int_{J_\go}(x-S_\go(\frac 12))^2 dP\\
&=\sum_{k=0}^{n-1}\sum_{\go\in I^{k}} \frac 13 \frac 1{75^{k}} \frac W{49^{a(n-k)}} +(\frac 2{75})^{n} V=\sum_{k=0}^{n-1} \frac 13 \Big(\frac 2{75}\Big)^k \frac W{49^{a(n-k)}} +(\frac 2{75})^{n} V\\
&=\sum_{k=0}^{n-2} \frac 13 \Big(\frac 2{75}\Big)^k \frac W{49^{n-k-1}}+ \frac 13 \Big(\frac 2{75}\Big)^{n-1} \frac W{49} +(\frac 2{75})^{n} V\\
&=\frac 13 \frac {W}{49}\frac{(\frac 2{75})^{n-1}-(\frac 1{49})^{n-1}}{\frac 2{75}-\frac 1{49}}+ \frac 13 \Big(\frac 2{75}\Big)^{n-1} \frac W{49} +(\frac 2{75})^{n} V\\
&=\Big(\frac 2{75}\Big)^{n-1}\Big( \frac 1 3\frac{W}{49} \frac {3675} {23} +\frac 13\frac W{49}+\frac 2{75} V\Big)-\Big(\frac 1{49}\Big)^{n-1} \frac 13 \frac {W}{49}  \frac {3675} {23}\end{align*}
yielding
\begin{equation*} \label{eq68} V_{F(n)}=\frac{69071}{6170325}\left(\frac{2}{75}\right)^{n-1}-\frac{3}{368}\Big(\frac 1{49}\Big)^{n-1}.
\end{equation*}
Thus, the proof of the proposition is complete.
\end{proof}

\subsection{Asymptotics for the $n$th quantization error $V_n(P)$}

In this subsection, we turn to the investigation of the quantization dimension $D(P)$ and the $D(P)$-dimensional quantization coefficient for the condensation measure $P.$ Notice that in this case $D(\gn)=\gb$, where $\gb=\frac{\log 2}{\log 7}$ which is the Hausdorff dimension of the limit set generated by the similarity maps $T_1$ and $T_2$ considered in this section.

\begin{theorem} \label{Th200}
$D(P)=\lim_{n\to \infty} \frac{2 \log n}{-\log V_n(P)}=\gk ; $ hence, the quantization dimension of $P$ exists and is equal to the critical value of the condensation system.
\end{theorem}
\begin{proof} Since $(\frac 13 (\frac 15)^2)^{\frac {\gk}{2+\gk}}=\frac{1}{2},\ \ \gk=\frac{2\log 2}{\log 75-\log 2}$.
For $n\in \D N$, $n\geq 4$,  let $\ell(n)$ be the least positive integer such that $F(\ell(n))\leq n<F(\ell(n)+1)$. Then,
$V_{F(\ell(n)+1)}<V_n\leq V_{F(\ell(n))}$. Thus, we have
\begin{align*}
\frac {2\log\left(F(\ell(n))\right)}{-\log\left(V_{F(\ell(n)+1)}\right)}< \frac {2\log n}{-\log V_n}< \frac {2\log\left(F(\ell(n)+1)\right)}{-\log\left(V_{F(\ell(n))}\right)}.
\end{align*}
Notice that when $n\to \infty$, then $\ell(n)\to \infty$. Recall that $F(\ell(n))=(\ell(n)+3) 2^{\ell(n)-1}$, and so by Proposition~\ref{prop720}, we have
\begin{align*}
&\lim_{\ell(n)\to\infty} \frac {2\log\left(F(\ell(n))\right)}{-\log\left(V_{F(\ell(n)+1)}\right)}=2 \lim_{\ell(n)\to\infty} \frac {\log(\ell(n)+3)+(\ell(n)-1)\log 2}{-\log\Big( \frac{69071}{6170325} \left(\frac{2}{75}\right)^{\ell(n)}-\frac{3}{368} \left(\frac{1}{49}\right)^{\ell(n)}\Big)}\\
&=2 \lim_{\ell(n)\to\infty} \frac{\frac{69071}{6170325} \left(\frac{2}{75}\right)^{\ell(n)}-\frac{3}{368} \left(\frac{1}{49}\right)^{\ell(n)}}{-\frac{69071}{6170325} \left(\frac{2}{75}\right)^{\ell(n)}\log \frac 2{75}+\frac{3}{368} \left(\frac{1}{49}\right)^{\ell(n)}\log \frac{1}{49}}\Big(\frac 1{\ell(n)+3} +\log 2 \Big)\\
&=\frac {2\log 2}{-\log \frac 2{75}}=\frac{2\log 2}{\log 75-\log 2}=\gk.
\end{align*}
Similarly, $\mathop{\lim}\limits_{\ell(n)\to\infty} \frac {2\log\left(F(\ell(n)+1)\right)}{-\log\left(V_{F(\ell(n))}\right)}=\gk$. Thus, $\gk\leq \liminf_n \frac{2\log n}{-\log V_n}\leq \limsup_n \frac{2\log n}{-\log V_n}\leq  \gk$ implying the fact that the quantization dimension of the measure $P$ exists and is equal to $\gk$.
\end{proof}

\begin{theorem} \label{Th300}
The $D(P)$-dimensional quantization coefficient for the condensation measure $P$ is infinity.
\end{theorem}

\begin{proof}
For $n\in \D N$, $n\geq 4$,  let $\ell(n)$ be the least positive integer such that $F(\ell(n))\leq n<F(\ell(n)+1)$. Then,
$V_{F(\ell(n)+1)}<V_n\leq V_{F(\ell(n))}$ implying $(F(\ell(n)))^{2/\gk}V_{F(\ell(n)+1)}<n^{2/\gk}V_n< (F(\ell(n)+1))^{2/\gk} V_{F(\ell(n))}$. As $\ell(n)\to \infty$ whenever $n\to \infty$, we have
\[\lim_{n\to \infty} \frac{F(\ell(n))}{F(\ell(n)+1)}=\lim_{n\to \infty} \frac{(\ell(n)+3)2^{\ell(n)-1}}{(\ell(n)+4)2^{\ell(n)}}=\frac 1{2}.\]
Next, since $\gk=\frac{2\log 2}{\log 75-\log 2}$, we have
\begin{align*}
&\lim_{n\to \infty} (F(\ell(n)))^{2/\gk} V_{F(\ell(n)+1)}=\frac 1{4^{1/\gk}} \lim_{n\to \infty} (F(\ell(n)+1))^{2/\gk} V_{F(\ell(n)+1)}\\
&=\frac 1{4^{1/\gk}} \lim_{n\to \infty} (\ell(n)+4)^{2/\gk} 2^{2\ell(n)/\gk} \Big(\frac{69071}{6170325} \left(\frac{2}{75}\right)^{\ell(n)}-\frac{3}{368} \left(\frac{1}{49}\right)^{\ell(n)}\Big)\\
&=\frac 1{4^{1/\gk}} \lim_{\ell(n)\to \infty} (\ell(n)+4)^{2/\gk} \Big(\frac {75}2\Big)^{\ell(n)} \Big(\frac{69071}{6170325} \left(\frac{2}{75}\right)^{\ell(n)}-\frac{3}{368} \left(\frac{1}{49}\right)^{\ell(n)}\Big)\\
&=\frac 1{4^{1/\gk}} \lim_{\ell(n)\to \infty} (\ell(n)+4)^{2/\gk} \Big(\frac{69071}{6170325} -\frac{3}{368} \left(\frac{75}{98}\right)^{\ell(n)}\Big)=\infty,
\end{align*} and similarly
\begin{align*}
\lim_{n\to \infty} (F(\ell(n)+1))^{2/\gk}V_{F(\ell(n))}=4^{1/\gk}\lim_{n\to \infty} (F(\ell(n)))^{2/\gk}V_{F(\ell(n))}=\infty
\end{align*}
yielding the fact that $\infty \leq \mathop{\liminf}\limits_{n\to\infty} n^{2/\gk} V_n(P)\leq \mathop{\limsup}\limits_{n\to \infty} n^{2/\gk}V_n(P)\leq \infty$, i.e., the $D(P)$-dimensional quantization coefficient for the condensation measure $P$ is infinity.
\end{proof}

Since $D(P)=\gk=\frac{2\log 2}{\log75-\log 2} > D(\gn)=\frac {\log 2}{\log 7} , $ the following proposition is also true.
\begin{prop} \label{prop3500}
$D(P)=\max \set{\gk, D(\gn)}$.

\end{prop}

\section{Condensation measures $P$ with self-similar measure $\gn$ satisfying \\ $D(\gn)>\gk$ and $D(\gn)=\gk$}\label{fourth}

In this section, we consider the quantization when $s=\frac{1}{5} $ and $s=\frac{\sqrt{6}}{15} . $  Since the proofs are the same as in the previous two cases, we will summarize the results and make some concluding remarks.  The condensation measures $P$ in these cases are generated by the systems $(\set{S_1, S_2}, (\frac 13, \frac 13, \frac 13), \gn)$,
where the measures $\gn$ are given by $\gn=\frac 12 \gn \circ T_1^{-1}+\frac 12\gn\circ T_2^{-1}$ with\\
\indent $T_1(x)=\frac{1}{5}x+\frac{8}{25}$, and $T_2(x)=\frac{1}{5}x+\frac{12}{25}$ in the case $s=\frac{1}{5}, $ and \\
\indent $T_1(x)=\frac{\sqrt{6}}{15}x+\frac{2}{5}-\frac{2 \sqrt{6}}{75}$, and $T_2(x)=\frac{\sqrt{6}}{15}x+\frac{3}{5}-\frac{\sqrt{6}}{25}$ in the case $s=\frac{\sqrt{6}}{15} . $
\medskip

\noindent Table 1 below outlines the information for the quantization of $P$ in each case.  The results in table are consequence of the particular facts that:
\begin{itemize}
\item[a)]  For $n\geq 3$, $\ga_n\ii J_1\neq \es$, $\ga_n \ii J_2\neq \es$, and $\ga_n\ii L\neq \es$, and $\ga_n$ does not contain any point from the open intervals $(\frac 15, \frac 25)$ and $(\frac 35, \frac 45)$.
\item[b)]  The sets $SF(n)$ and $SF^\ast (n)$, and the order relation $\succ$ among the elements of them are also defined the same way as in Section~\ref{second}.  Furthermore,
\begin{itemize}
\item[(i)] $ \cdots \succ  S(4)\succ S(3)\succ S(2)\succ S(0) \succ S^{(2)}(0)\succ S(1).$
\item[(ii)] Although the canonical sequences $\set{a(n)}_{n\geq 1}$ and  $\set{F(n)}_{n\geq 1}$ are identical, the order in the elements of $SF^\ast(n)$ are different.
\item[(iii)]  For $A, B\in SF^\ast(n)$ with $A\succ B$, the distortion error due to the set $ (SF^\ast(n)\setminus A)\uu A^{(2)}\uu B$ is less than the distortion error due to the set $ (SF^\ast (n)\setminus B)\uu B^{(2)}\uu A$.
\end{itemize}
\end{itemize}

\medskip

\begin{table}
\caption{}
\label{tab: table1}
\begin{tabular}{ |p{3.4cm}||p{6.1cm}|p{6.1cm}|  }
 \hline
 & Case $s=\frac{1}{5} $ & Case $s=\frac{\sqrt{6}}{15} $ \\
\hline
$E(\nu), \ V(\nu) $ & $\frac{1}{2}, \frac{1}{150} $    &   $\frac{1}{2},\ \frac{77-10 \sqrt{6}}{7500} $ \\
\hline
$E(P), V(P) $ & $\frac12, \ \frac{49}{438} $ & $\frac12, \  \frac{2413-10 \sqrt{6}}{21316} $ \\
\hline
$\alpha_1,\ V_1$ & $\{\frac12 \}, \ \frac{49}{438} $ & $\alpha_1$ same, $  \frac{2413-10 \sqrt{6}}{21316} $ \\
\hline
$\alpha_2,\ V_2 $ & $\{\frac{31}{150}, \frac{119}{150} \} \ \frac{21211}{821250}$ & $\{\frac{\sqrt{6}+90}{450},\frac{30-\sqrt{6}}{450}\}, \  \frac{155430 \sqrt{6}+4167521}{179853750}$  \\
\hline
$\alpha_3,\ V_3 $ & $\{\frac{1}{10}, \frac{1}{2}, \frac{9}{10} \}, \ \frac{19}{3650}$  & $\alpha_3$ same,$ \frac{10447-750 \sqrt{6}}{1598700} $ \\
\hline
$\alpha_4,\ V_4 $ & $\{S_1(\frac 12), T_1(\frac 12), T_2(\frac 12), S_2(\frac 12)\}, \ \frac{841}{273750}$  & $\alpha_4$ same, $ \frac{93298-740 \sqrt{6}}{29975625}$  \\
\hline
$a(n)$  & 1 if $n=1, \ n-1$ for $n>1$  & same \\
\hline
$F(n)$  & $(n+3) 2^{n-1} $  & same \\
\hline
$V_{F(n)}$  & $\frac 1{25^{n-1}}-\frac{164}{45625}\Big(\frac 2{75} \Big)^{n-1} $ if  $n\geq 2  $  & $(\frac{2}{75})^{n-1}(\frac{V(\nu)}{3}n+\frac{2}{75} V(P))$ for $n\geq 2$\\
\hline
$\kappa $  & $\frac{2\log 2}{\log75-\log 2} $  & $\frac{2\log 2}{\log75-\log 2} $ \\
\hline
$D(\nu)$  & $\log 2 / \log 5 $ & $\frac{2\log 2}{\log75-\log 2} $ \\
\hline
$D(P)$  & $\log 2 / \log 5 $  & $\frac{2\log 2}{\log75-\log 2} $ \\
\hline
Quant. coefficients  & $\infty $  & $\infty $ \\
\hline
\end{tabular}
\end{table}

\medskip

Since, from
$(\frac 13 (\frac 15)^2)^{\frac {\gk}{2+\gk}}=\frac{1}{2}$ we have $\gk=\frac{2\log 2}{\log75-\log 2} $ as the critical value, it follows that for these condensation systems $D(P)=\max \set{\gk, D(\gn)}$ as well.

\bigskip

\subsection{Concluding Remarks} \label{conr}


For the condensation system $(\set{S_1, S_2}, (\frac 13, \frac 13, \frac 13), \gn)$ generating the condensation measure $P$, where $S_1, S_2$ are the similarity mappings
and $\gn$ is a self-similar measure defined by $\nu=\frac 12 \nu\circ T_1^{-1}+\frac 12 \nu\circ T_2^{-1}$ with $T_1(x)=sx+(1-s)\frac 2 5 $ and $T_2(x)=sx+(1-s)\frac 3 5 , \ 0<s<\frac 13 , $ by Theorem~\ref{Th3}, Theorem~\ref{Th300}, and Table 2, we see that when $s= \frac{1}{7}, \frac{\sqrt{6}}{15}, \frac 15$, the $D(P)$ dimensional lower quantization coefficient for the condensation measure $P$ is infinity; on the other hand, if $s=\frac 13$, then the $D(P)$-dimensional lower and upper quantization coefficients are finite, positive and unequal. Notice that  $\frac{1}{7}< \frac{\sqrt{6}}{15}< \frac 15<\frac 1 3$. Thus, it is worthwhile to investigate the least upper bound of $s$ for which the  $D(P)$ dimensional lower quantization coefficient for the condensation measure $P$ is infinity. Such a problem still remains open.

Observe that, for $s= \frac{1}{7} $ and $s= \frac{\sqrt{6}}{15},$ although the quantization coefficients are infinity, whereas the relations between $D(P)$ and $D(\nu)$ are not the same.  Again, it is worthwhile to investigate the least upper bound of $s$ for which the $D(P)=D(\nu)$ while $D(P)$-dimensional lower quantization coefficient for the condensation measure $P$ is infinity.

\begin{table}
\caption{}
\label{tab: table2}
\begin{tabular}{ |p{4cm}||p{3.5cm}|p{3.5cm}|p{4cm}|  }
 \hline
 Contracting factor & Quant. dimensions  & Quant. coefficients & $D(\nu)$ vs critical value \\
 \hline
$s=\frac{1}{3} $  &  $D(\nu)=D(P)$   & finite, positive  &  $D(\nu )>\kappa$ \\
\hline
$s=\frac{1}{5} $   &  $D(\nu)=D(P)$   & infinite   &  $D(\nu )>\kappa$ \\
\hline
$s=\frac{\sqrt{6}}{15} $  & $D(\nu)=D(P)$  & infinite  & $D(\nu ) = \kappa$ \\
\hline
$s=\frac{1}{7} $    & $D(\nu)<D(P)$  & infinite  &  $D(\nu )<\kappa$\\
 \hline
\end{tabular}
\end{table}

\medskip

{\bf Data Availability Statement.}  Data sharing not applicable to this article as no datasets were generated or analyzed during the current study.  However, detailed calculations for the items summarized in Table 1 are available from the corresponding author upon request.

\end{document}